\documentclass[11pt]{article}
\usepackage{etex}
\usepackage{mathtools}
\usepackage{enumitem}

\usepackage{amsmath,amsthm}
\usepackage{csquotes}
\usepackage{todonotes}

\usepackage{amssymb}

\usepackage{amssymb}

\usepackage{graphicx}

\usepackage{subfig}
\usepackage[final]{showkeys}

\usepackage{etoolbox}
\usepackage{relsize}

\usepackage{wasysym}

\usepackage{mathrsfs}

\usepackage{mathpazo}

\usepackage[titletoc]{appendix}
\usepackage[doc]{optional}

\usepackage{soul}

\usepackage{colortbl,booktabs,sectsty,multirow}
\usepackage{xcolor}

\usepackage{cancel}

\usepackage{empheq}

\definecolor{myblue}{rgb}{.8, .8, 1}
  \newcommand*\mybluebox[1]{
    \colorbox{myblue}{\hspace{1em}#1\hspace{1em}}}

\usepackage[obeyspaces,hyphens,spaces]{url}

\usepackage{hyperref}
\hypersetup{
    colorlinks=true,
    linkcolor=blue,
    citecolor=blue,
    filecolor=magenta,
    urlcolor=cyan
}

\usepackage[
  open,
  openlevel=2,
  atend,
  numbered
]{bookmark}

\usepackage[capitalize, nameinlink, noabbrev]{cleveref}
\crefname{equation}{}{}
\crefname{chapter}{Chapter}{Chapters}
\crefname{item}{item}{items}
\crefname{figure}{Figure}{Figures}
\crefname{theorem}{Theorem}{Theorems}
\crefname{lemma}{Lemma}{Lemmas}
\crefname{proposition}{Proposition}{Propositions}
\crefname{corollary}{Corollary}{Corollarys}
\crefname{definition}{Definition}{Definitions}
\crefname{fact}{Fact}{Facts}
\crefname{example}{Example}{Examples}
\crefname{algorithm}{Algorithm}{Algorithms}
\crefname{remark}{Remark}{Remarks}
\crefname{note}{Note}{Notes}
\crefname{notation}{Notation}{Notations}
\crefname{case}{Case}{Cases}
\crefname{exercise}{Exercise}{Exercises}
\crefname{question}{Question}{Questions}
\crefname{claim}{Claim}{Claims}
\crefname{enumi}{}{}

\usepackage[top= 2cm, bottom = 2 cm, left = 2.2 cm, right= 2.2 cm]{geometry}

\usepackage{float}

\usepackage{pgf}

\usepackage{array}
\usepackage{tabu}

\allowdisplaybreaks

\numberwithin{equation}{section}

\theoremstyle{plain}
\newtheorem{theorem}{Theorem}[section]

\newtheorem{corollary}[theorem]{Corollary}
\newtheorem{fact}[theorem]{Fact}
\newtheorem{lemma}[theorem]{Lemma}
\newtheorem{proposition}[theorem]{Proposition}

\theoremstyle{definition}
\newtheorem{definition}[theorem]{Definition}
\newtheorem{example}[theorem]{Example}

\newcommand{\argmin}{\ensuremath{\operatorname{argmin}}}
\newcommand{\inte}{\ensuremath{\operatorname{int}}}
\newcommand{\bd}{\ensuremath{\operatorname{bd}}}

\newcommand{\aff}{\ensuremath{\operatorname{aff} \,}}

\newcommand{\spn}{\ensuremath{{\operatorname{span} \,}}}

\newcommand{\dom}{\ensuremath{\operatorname{dom}}}

\newcommand{\Id}{\ensuremath{\operatorname{Id}}}

\newcommand{\C}{\ensuremath{\operatorname{C}}}
\newcommand{\D}{\ensuremath{\operatorname{D}}}

\newcommand{\Pro}{\ensuremath{\operatorname{P}}}

\newcommand{\I}{\ensuremath{\operatorname{I}}}

\newcommand{\CCO}[1]{CC{#1}}

\providecommand{\abs}[1]{\lvert#1\rvert}
\providecommand{\norm}[1]{\lVert#1\rVert}

\providecommand{\innp}[1]{\langle#1\rangle}
\providecommand{\Innp}[1]{\Big\langle#1\Big\rangle}

\begin{document}

\title{ \sffamily
	Bregman circumcenters: basic theory
}

\author{
         Hui\ Ouyang\thanks{
                 Mathematics, University of British Columbia, Kelowna, B.C.\ V1V~1V7, Canada.
                 E-mail: \href{mailto:hui.ouyang@alumni.ubc.ca}{\texttt{hui.ouyang@alumni.ubc.ca}}.}~
         and Xianfu\ Wang\thanks{
                 Mathematics, University of British Columbia, Kelowna, B.C.\ V1V~1V7, Canada.
                 E-mail: \href{mailto:shawn.wang@ubc.ca}{\texttt{shawn.wang@ubc.ca}}.}
                 }

\date{April 5, 2021}
\maketitle

\begin{abstract}
	\noindent
Circumcenters play an important role in the design and analysis of accelerating
various iterative methods in optimization.
In this work,
we propose Bregman (pseudo-)circumcenters associated with finite sets.  We show the existence
and give explicit formulae for the unique backward and
	 forward Bregman pseudo-circumcenters of finite sets. Moreover, we use duality to establish
connections between backward and forward Bregman (pseudo-)circumcenters.
Various examples are presented to illustrate the backward and forward Bregman (pseudo-)circumcenters of finite sets.
Our general framework for circumcenters paves the way for the development of
accelerating iterative methods by Bregman circumcenters.
\end{abstract}
	
	{\small
		\noindent
		{\bfseries 2020 Mathematics Subject Classification:}
		{
			Primary 90C48, 47H04, 47H05;
			Secondary 90C25,  52A41.
		}
		
		\noindent{\bfseries Keywords:}
		Bregman distance, Legendre function, backward Bregman projection, forward Bregman projection,   backward Bregman (pseudo-)circumcenter,  forward Bregman (pseudo-)circumcenter.
	}
\section{Introduction}

Circumcenter is a classical concept in geometry. Recently, circumcenters have been used to accelerate iterative methods, such as the Douglas-Rachford method and method of alternating projections, in optimization; see e.g., \cite{BBCS2020CRMbetter, BBCS2017, BBCS2018, BBCS2019, BBCS2020ConvexFeasibility}. Compared with the classic Douglas-Rachford method and the method of alternating projections, the circumcentered methods present much better performance for solving the best approximation problems and feasibility problems.
In \cite{BOyW2018}, we provided a systematic study on the circumcenter of a finite set in a Hilbert space.
This allowed us to investigate circumcentered
methods of nonexpansive mappings, isometries, and best approximation mappings; see, e.g., \cite{BOyW2019Isometry, BOyW2019LinearConvergence, BOyW2020BAM}.
Bregman distances have also been widely studied in optimization;
see \cite{censor, BBC2003, BC2003, BD2002, reich, wen, laude} and references therein.
A natural question is: Can one define circucmcenters using Bregman distances and use them to accelerate
iterative methods? However, up to now a study of circumcenters in
the framework of Bregman distances is still missing in the literature.

\emph{In this work, under general Bregman distances, we introduce appropriate definitions of  circumcenters of finitely many points in a Hilbert space,  investigate the existence and uniqueness of  circumcenters of finite sets,  and present explicit formulae for the unique circumcenters. One of the distinguished features is that while the classical circumcenter might
fail to exist the Bregman (pseudo-)circumcenters can exist. Our work sets up the theoretical foundation for utilizing Bregman circumcenters to accelerate
iterative methods in optimization. 
}

In general, the Bregman distances are neither symmetric nor
full domain, these cause many technical challenges. On the one hand, 
we have to introduce backward Bregman (pseudo-)circumcenters
and forward Bregman (pseudo-)circumcenters; on the other hand, appropriate affine subspaces are 
needed to explore the uniqueness
of Bregman circumcenters.
Our main results in this work are the following:
\begin{itemize}
	\item[\textbf{R1:}] \cref{theor:affine:character:P} characterizes backward and forward Bregman projectors onto affine subspaces.
		
	\item[\textbf{R2:}] We provide sufficient conditions for the existence of  backward and forward Bregman circumcenters of finite sets  and give corresponding circumcenters in  \Cref{theorem:formualCCS:Pleft,theorem:forwardCCS} by
using the backward and forward Bregman projections onto affine subspaces.
Moreover, the unique backward and forward Bregman pseudo-circumcenters  are characterized in  \Cref{theorem:formualCCS,theorem:psuCCS:forward} by using the Eucliean projections onto affine subspaces, respectively.
	
	\item[\textbf{R3:}] Some dual expressions of the backward and forward Bregman (pseudo-)circumcenters 
are demonstrated in \cref{theor:CCS:Rel}.
\end{itemize}

The paper is organized as follows. Some fundamental results on Bregman distances and projections are presented in
\cref{sec:Preliminaries} for subsequent usage.
In \Cref{sec:BackwardBregmancircumcenters,sec:ForwardBregmancircumcenters}, we systematically
investigate backward and forward Bregman (pseudo-)circumcenters, respectively. 
In particular, we give backward and forward Bregman circumcenters of finite sets via Bregman projections; show 
the uniqueness of backward and forward Bregman pseudo-circumcenters; state equivalent expressions of  backward and forward Bregman pseudo-circumcenters; and provide concrete examples of backward and forward Bregman (pseudo-)circumcenters.
In \cref{sec:Miscellaneous}, we establish some duality correspondences between backward and forward Bregman pseudo-circumcenters.
Section~\ref{s:compare} compares the Bregman (pseudo-)circumcenters with the classical circumcenter by examples. While the
classical circumcenter might not exist the Bregman circumcenters exist.
 Section~\ref{s:conclude} finishes the paper.

\section{Preliminaries} \label{sec:Preliminaries}
In this section, we review some results on Bregman distances and characterize forward and backward Bregman 
projections onto affine subspaces, which are essential to our later analysis.

\subsection*{Notation}
Throughout the work, we assume that $\mathbb{N}=\{0,1,2,\ldots\}$ and $\{ m,  n \} \subseteq  \mathbb{N} \smallsetminus \{0\}$, and that
\begin{empheq}[box=\mybluebox]{equation*}
\mathcal{H}    \text{ is a real Hilbert space  with inner product } \innp{\cdot, \cdot} \text{ and induced norm } \norm{\cdot}.	
\end{empheq}
$\Gamma_{0} (\mathcal{H}) $ is the set of proper closed convex functions 
from $\mathcal{H}$ to $\left]-\infty, +\infty\right]$.
Let $f:\mathcal{H} \to \left]-\infty, +\infty\right]$ be proper.  The \emph{domain}  
(\emph{conjugate function, gradient},  respectively) of $f$ is 
denoted by $\dom f$ ($f^{*}$,    $\nabla f$,   respectively).
 Let $C$ be a nonempty subset of $\mathcal{H}$.  Its \emph{interior} and \emph{boundary} are abbreviated by $\inte C$  and $\bd C$, respectively.
 $C$ is an \emph{affine subspace} of
 $\mathcal{H}$ if $C \neq \varnothing$ and $(\forall \rho\in\mathbb{R})$ $\rho
 C + (1-\rho)C = C$. The smallest affine subspace of $\mathcal{H}$ containing $C$ is
  denoted by $\aff C$ and called the \emph{affine hull} of $C$. The \emph{orthogonal
 	complement of $C$} is the set $ C^{\perp} :=\{x \in \mathcal{H}~:~ (\forall y \in C) ~\innp{x,y}=0\}$. The \emph{best approximation operator} (or \emph{projector}) onto $C$ under the Euclidean distance is denoted by $\Pro_{C}$, that is, $(\forall x \in \mathcal{H} )$ $\Pro_{C}x :=  \argmin_{y \in C} \norm{x-y}$.
 Given  the points  $ a_{1}, \ldots, a_{m} $ in $ \mathcal{H}$, the \emph{Gram matrix} $G(a_{1}, \ldots, a_{m})$ is defined as
 \begin{align*}
 G(a_{1}, a_{2}, \ldots, a_{m}) :=
 \begin{pmatrix}
 \norm{a_{1}}^{2} &\innp{a_{1},a_{2}} & \cdots & \innp{a_{1}, a_{m}}  \\
  \innp{a_{2},a_{1}} & \norm{a_{2}}^{2} & \cdots & \innp{a_{2},a_{m}} \\
 \vdots & \vdots & ~~& \vdots \\
 \innp{a_{m},a_{1}} & \innp{a_{m},a_{2}} & \cdots & \norm{a_{m}}^{2} \\
 \end{pmatrix}.
 \end{align*}
 For every $x \in \mathcal{H}$ and $\delta \in \mathbb{R}_{++}$,  $B [x; \delta] $ is the \emph{closed ball with center at $x$ and with radius $\delta$}. Let $A : \mathcal{H} \to 2^{\mathcal{H}}$ and let $x \in \mathcal{H}$. Then $A$ is \emph{locally bounded at $x$} if there exists $\delta \in \mathbb{R}_{++}$ such that $A(B[x;\delta] )$ is bounded.
  For convenience, if $A(x)$ is a singleton for some $x \in \mathcal{H}$, by a slight abuse of notation we allow
  $A(x)$ to stand for  its unique element.
For other notation not explicitly defined here, we refer the reader to \cite{BC2017}.

\subsection*{Legendre functions and Bregman distances}
Legendre functions are instrumental for our analysis. 
\begin{definition} {\rm \cite[Definition~5.2 and Theorem~5.6]{BBC2001}} \label{def:Legendre}
	Suppose  that $f \in \Gamma_{0} (\mathcal{H}) $.  We say $f$ is:
	\begin{enumerate}
		\item  \emph{essentially smooth},
$\dom \partial f =\inte \dom f \neq \varnothing$, $f$ is G\^ateaux differentiable on $\inte \dom f$, and $\norm{\nabla f (x_{k})} \to +\infty$, for every sequence $(x_{k})_{k \in \mathbb{N}}$ in $\inte \dom f$ converging to some point in $\bd \dom f$.
		\item  \label{def:Legendre:convex} \emph{essentially strictly convex}, if $(\partial f)^{-1}$ is locally bounded on its domain and $f$ is strictly convex on every convex subset of $\dom \partial f$.
		\item \emph{Legendre},  if $f$ is both essentially smooth and essentially strictly convex.
	\end{enumerate}
\end{definition}

\begin{fact}{\rm \cite[Lemma~7.3(vii)]{BBC2001}}  \label{fact:PropertD}
	Suppose that $f \in \Gamma_{0} (\mathcal{H}) $  with $\inte \dom f \neq \varnothing$, that $f$ is G\^ateaux differentiable  on $\inte \dom f$, and that $f$ is essentially strictly convex.
	Let $x$ and $y$ be in $ \inte \dom f$. Then  $D_{f}(x, y) = D_{f^{*}}(\nabla f (y),\nabla f (x))$.
\end{fact}

\begin{fact} {\rm \cite[Corollary~5.5 and Theorem~5.10]{BBC2001}}  \label{fact:nablaf:nablaf*:id}
	Suppose  that $f \in \Gamma_{0} (\mathcal{H}) $.  Then $f$ is Legendre if and only if $f^{*}$ is. In this case,  the gradient mapping $\nabla f: \inte \dom f \to \inte \dom f^{*}$
	is bijective, with inverse $(\nabla f)^{-1} =\nabla f^{*} : \inte \dom f^{*} \to \inte \dom f$. Moreover, the gradient mappings $\nabla f$, $\nabla f^{*}$ are both norm-to-weak continuous and locally bounded on their respective domains.
\end{fact}

\begin{definition} {\rm \cite[Definitions~7.1 and 7.7]{BBC2001}} \label{defn:BregmanDistance}
	Suppose that $f \in \Gamma_{0} (\mathcal{H}) $  with $\inte \dom f \neq \varnothing$ and that $f$ is G\^ateaux differentiable  on $\inte \dom f$. The  \emph{Bregman distance $\D_{f}$ associated with $f$} is defined by
	\begin{align*}
\D_{f}: \mathcal{H} \times \mathcal{H} \to \left[0, +\infty\right] : (x,y) \mapsto \begin{cases}
f(x) -f(y) -\innp{\nabla f(y), x-y}, \quad &\text{if } y \in  \inte \dom f;\\
+\infty,  \quad &\text{otherwise}.
\end{cases}
\end{align*}
	Moreover, let $C$ be a nonempty subset of $\mathcal{H}$. For every $(x,y) \in \dom f \times \inte \dom f$, define the \emph{backward Bregman projection}  of $y$ onto $C$ and  \emph{forward  Bregman projection} of $x$ onto $C$,  respectively, as
\begin{align*}
&\overleftarrow{\Pro}^{f}_{C}(y):= \left\{ u \in C \cap \dom f ~:~  (\forall c \in \C) \D_{f} \left( u,y \right) \leq \D_{f} (c,y)   \right\}, \text{and}\\
&\overrightarrow{\Pro}^{f}_{C}(x) := \left\{  v \in C \cap \inte \dom f ~:~ (\forall c \in \C) \D_{f} \left( x,v \right) \leq \D_{f} (x, c)  \right\}.
\end{align*}
\end{definition}

Clearly, one recovers the Euclidean distance $\D :   \mathcal{H} \times \mathcal{H} \to \left[0, +\infty\right] :  (x,y) \mapsto \frac{1}{2} \norm{x-y}^{2}$ by setting $f = \frac{1}{2} \norm{\cdot}^{2}$ in   \cref{defn:BregmanDistance}.

\subsection*{Bregman  projections}
 \Cref{fact:charac:PleftCf,fact:charac:PrightCf} play an essential role 
 for the characterizations of Bregman projections onto affine subspaces.

\begin{fact}  \label{fact:charac:PleftCf}
		Suppose  that $f \in \Gamma_{0} (\mathcal{H}) $   is Legendre, that $C$ is a closed convex subset of $\mathcal{H}$ with $C \cap \inte \dom f \neq \varnothing$, and that $y \in \inte \dom f $. Then  for every $z \in \mathcal{H}$,
		\begin{align} \label{eq:fact:charac:PleftCf}
	z=	\overleftarrow{\Pro}^{f}_{C}(y) \Leftrightarrow \left[ z \in C \cap \inte \dom f \quad \text{and} \quad (\forall c \in C)~ \Innp{c-z, \nabla f (y) -\nabla f (z)} \leq 0 \right];
		\end{align}
equivalently,
		\begin{align}\label{eq:fact:charac:PleftCf:D}
		z=	\overleftarrow{\Pro}^{f}_{C}(y) \Leftrightarrow \left[  z  \in C \cap \inte \dom f \quad \text{and} \quad (\forall c \in C)~ \D_{f} (c,y) \geq \D_{f} (c,z)  +\D_{f} (z,y) \right].
		\end{align}
	\end{fact}
\begin{proof}
	Note that although \cite{BB1997Legendre} is on $\mathbb{R}^{n}$, the proof of \cite[Proposition~3.16]{BB1997Legendre} works in the Hilbert space as well.  Hence, mimic  the proof of \cite[Proposition~3.16]{BB1997Legendre} and apply     \cite[Corollary~7.9]{BBC2001} to  obtain \cref{eq:fact:charac:PleftCf}.
	
	In addition, in view of \cref{eq:fact:charac:PleftCf} and \cref{defn:BregmanDistance}, we have
	\begin{align*}
	z=	\overleftarrow{\Pro}^{f}_{C}(y) \Rightarrow \left[  z  \in C \cap \inte \dom f \quad \text{and} \quad (\forall c \in C)~ \D_{f} (c,y) \geq \D_{f} (c,z)  +\D_{f} (z,y) \right].
	\end{align*}
	Furthermore, assume that  $z  \in C \cap \inte \dom f$ and $ (\forall c \in C)$ $\D_{f} (c,y) \geq \D_{f} (c,z)  +\D_{f} (z,y) $. Then
	\begin{align*}
(\forall c \in C) \quad \D_{f} (c,y) \geq \D_{f} (c,z)  +\D_{f} (z,y) \geq  \D_{f} (z,y),
	\end{align*}
which, by  \cref{defn:BregmanDistance}, implies that 	$z=	\overleftarrow{\Pro}^{f}_{C}(y)$. Altogether, \cref{eq:fact:charac:PleftCf:D} is true.
\end{proof}

\begin{definition} {\rm \cite[Definition~2.4]{BC2003}}
	Suppose that   $f \in \Gamma_{0} (\mathcal{H}) $  is  Legendre  such that $\dom f^{*}$ is open.  We say \emph{the function $f$ allows forward Bregman projections} if it satisfies the following properties.
	\begin{enumerate}
		\item $\nabla^{2} f$ exists and is continuous on $\inte \dom f$.
		\item $\D_{f}$ is convex on $(\inte \dom f)^{2}$.
		\item For every $x \in \inte \dom f$, $D_{f}(x, \cdot)$ is strictly convex on $\inte \dom f$.
	\end{enumerate}
\end{definition}

\begin{fact} {\rm \cite[Fact~2.6]{BC2003}}  \label{fact:charac:PrightCf}
	Suppose that $\mathcal{H} =\mathbb{R}^{n}$ and  $f \in \Gamma_{0} (\mathcal{H}) $ is  Legendre  such that $\dom f^{*}$ is open, that $f$ allows forward Bregman projections, and that  $C $ is a closed convex subset of $\mathcal{H}$ with $C \cap \inte \dom f \neq \varnothing$. Then for every $y \in \inte \dom f$,
the forward Bregman projection  of $y$ onto $C$ uniquely exists,
	 and for every $z \in \mathcal{H}$,
	\begin{align*}
	z =\overrightarrow{\Pro}^{f}_{C}(y) \Leftrightarrow \left[ z \in C \cap \inte \dom f \quad \text{and} \quad (\forall c \in C) ~ \Innp{c-z, \nabla^{2} f (z) (y - z)} \leq 0 \right];
	\end{align*}
	equivalently,
	\begin{align*}
z =\overrightarrow{\Pro}^{f}_{C}(y) \Leftrightarrow \left[ z   \in C \cap \inte \dom f \quad \text{and} \quad (\forall c \in C)~ \D_{f} (c,y) \geq \D_{f} (c,z)  +\D_{\D_{f}} \left((c,c), (y , z)\right) \right].
	\end{align*}
	Moreover, the \emph{forward Bregman projector} $\overrightarrow{\Pro}^{f}_{C}$ is continuous on $\inte \dom f$.
\end{fact}

Note that not every Legendre function allows forward Bregman projections and
that the backward and forward Bregman projections are different notions (see, e.g.,  \cite{BB1997Legendre}, \cite{BC2003} and \cite{BD2002}).
Below  we list particular functions satisfying conditions required by \cref{fact:charac:PleftCf} or \cref{fact:charac:PrightCf}.
\begin{fact} {\rm \cite[Examples~2.1 and 2.7]{BC2003}} \label{fact:examplef}
Suppose that $\mathcal{H} =\mathbb{R}^{n}$. 	Denote by  $(\forall x \in \mathcal{H} )$ $x= (x_{i})^{n}_{i=1} $. Then
the following functions  are  Legendre  such that the domain of their conjugates are open. Moreover, the energy, negative entropy,  and Fermi-Dirac entropy allow forward Bregman projections.\footnotemark
\footnotetext{Here and elsewhere, we use the convention that $0 \ln (0) =0$.}
\begin{enumerate}
	\item $($energy$)$ $f :x \mapsto \frac{1}{2} \norm{x}^{2} =\frac{1}{2}  \sum^{n}_{i=1} \abs{x_{i}}^{2}$, with $\dom f =\mathbb{R}^{n}$.
	\item $($negative entropy$)$ $f :x \mapsto   \sum^{n}_{i=1} x_{i} \ln (x_{i}) -x_{i}$, with $\dom f =\left[0, +\infty\right[^{n} $.
	\item $($Fermi-Dirac entropy$)$ $f :x \mapsto   \sum^{n}_{i=1} x_{i} \ln  (x_{i}) + (1-x_{i} ) \ln(1-x_{i} )$, with $\dom f =\left[0, 1\right]^{n} $.
		\item $($Burg entropy$)$ $f :x \mapsto   -\sum^{n}_{i=1}   \ln (x_{i})  $, with $\dom f =\left]0, +\infty\right[^{n} $.
	\item   $f :x \mapsto  - \sum^{n}_{i=1} \sqrt{x_{i} }$, with $\dom f =\left[0, +\infty\right[^{n} $.
\end{enumerate}
\end{fact}

Bregman projections onto affine subspaces are characterized in the following result.
Notice that \cref{theor:affine:character:P}\cref{theor:affine:character:P:Lback} with
$f(x)= \sum^{n}_{i=1} x_{i}\log x_{i}$, where $x:= (x_{1}, \ldots, x_{n}) \in \mathbb{R}^{n}$, reduces to
\cite[Corollary~2.1]{Teboulle1992}.
\begin{theorem} \label{theor:affine:character:P}
	Suppose that $f \in \Gamma_{0} (\mathcal{H}) $ is  Legendre,    that   $U $ and $L$ are, respectively, closed affine  and linear subspaces of $\mathcal{H}$ with $U \cap \inte \dom f \neq \varnothing$ and $L \cap \inte \dom f \neq \varnothing$.  Let $y \in \inte \dom f$ and $z \in \mathcal{H}$.  Then the following assertions hold.
	\begin{enumerate}
		\item \label{theor:affine:character:PleftCf} The following three statements are equivalent.
		\begin{enumerate}
			\item \label{theor:affine:character:PleftCf:=}$z = 	\overleftarrow{\Pro}^{f}_{U} (y) $.
			\item \label{theor:affine:character:PleftCf:innp}$z \in U \cap \inte \dom f $ and $(\forall u \in U) $ $ \innp{\nabla f(y) -\nabla f ( z), u- z } = 0 $.
			\item \label{theor:affine:character:PleftCf:D}$z \in U \cap \inte \dom f $ and $(\forall u \in U) $ $\D_{f} (u,y) = \D_{f} (u,z)  +\D_{f} (z,y)$.
		\end{enumerate}
		\item  \label{theor:affine:character:P:Lback} $z = 	\overleftarrow{\Pro}^{f}_{L} (y)  \Leftrightarrow \left[  z \in L \cap \inte \dom f ~\text{and} ~ \left( \nabla f(y) -\nabla f ( z) \right) \in L^{\perp} \right].$
		\item \label{theor:affine:character:PrightCf} Suppose that  $\mathcal{H} =\mathbb{R}^{n}$, that  $\dom f^{*}$ is open, and that $f$ allows forward Bregman projections. Then the following three statements are equivalent.
		\begin{enumerate}
			\item \label{theor:affine:character:PrightCf:=} $z =\overrightarrow{\Pro}^{f}_{U}(y)$.
			\item \label{theor:affine:character:PrightCf:innp} $z \in U \cap \inte \dom f $ and $(\forall u \in U) $  $ \innp{u-z, \nabla^{2} f (z) (y - z)} = 0 $.
			\item \label{theor:affine:character:PrightCf:D} $z \in U \cap \inte \dom f $ and $(\forall u \in U) $  $\D_{f} (u,y) = \D_{f} (u,z)  +\D_{\D_{f}} \left((u,u), (y , z)\right)$.
		\end{enumerate}
		\item \label{theor:affine:character:P:Lfor} Suppose that  $\mathcal{H}  =\mathbb{R}^{n}$, that  $\dom f^{*}$ is open, and that $f$ allows forward Bregman projections. Then $z =\overrightarrow{\Pro}^{f}_{L}(y) \Leftrightarrow \left[ z \in L \cap \inte \dom f ~ \text{and} ~  \nabla^{2} f (z) (y - z) \in L^{\perp} \right]. $
	\end{enumerate}
	
\end{theorem}

\begin{proof}
	\cref{theor:affine:character:PleftCf}:
	Suppose that $z = 	\overleftarrow{\Pro}^{f}_{U} (y) $.
	According to \cref{fact:charac:PleftCf}, $z  \in U \cap \inte \dom f $ and
	\begin{align} \label{eq:PCf:leq}
	(\forall u \in U ) \quad \innp{\nabla f(y) -\nabla f ( z ), u-z } \leq 0.
	\end{align}
	Because $z \in U$ and $U $ is affine, we have that  $(\forall u \in U)$  $z +(z-u) =2z-u \in U $. Hence,
	\begin{align} \label{eq:PCf:geq}
	(\forall u \in U ) \quad \innp{\nabla f(y) -\nabla f ( z ), z-u } =\innp{\nabla f(y) -\nabla f ( z ), z +(z-u)-z } \stackrel{\cref{eq:PCf:leq}}{\leq} 0.
	\end{align}
	Combine \cref{eq:PCf:leq} and \cref{eq:PCf:geq} to yield the required \cref{theor:affine:character:PleftCf:=}$\Leftrightarrow$\cref{theor:affine:character:PleftCf:innp}.
	
	On the other hand, similarly with the last part of our proof of \cref{fact:charac:PleftCf},  	\cref{theor:affine:character:PleftCf:innp} $\Leftrightarrow$ \cref{theor:affine:character:PleftCf:D} follows  from   \cref{defn:BregmanDistance} and \cref{theor:affine:character:PleftCf:=}$\Leftrightarrow$\cref{theor:affine:character:PleftCf:innp}.
	
	\cref{theor:affine:character:P:Lback}: This is clear from 	\cref{theor:affine:character:PleftCf:=} $\Leftrightarrow$ \cref{theor:affine:character:PleftCf:innp} above.
	
	\cref{theor:affine:character:PrightCf}: The proof of \cref{theor:affine:character:PrightCf:=} $\Leftrightarrow $ \cref{theor:affine:character:PrightCf:innp} is similar to the proof of \cref{theor:affine:character:PleftCf:=} $ \Leftrightarrow$ \cref{theor:affine:character:PleftCf:innp}, while this time we apply  \cref{fact:charac:PrightCf} instead of \cref{fact:charac:PleftCf}.
	
	Let $u \in U$.
By  \cref{defn:BregmanDistance},
\begin{align*}
&\D_{f} (u,y) = \D_{f} (u,z)  +\D_{\D_{f}} \left((u,u), (y , z)\right) \\
\Leftrightarrow & \D_{f} (u,y) = \D_{f} (u,z)  +\D_{f}(u,u) -\D_{f}(y , z) - \Innp{ \Big(\nabla f (y) - \nabla f(z), -\nabla f^{2} (z)(y-z)  \Big), (u-y,u-z)}\\
\Leftrightarrow &  \D_{f} (u,y) = \D_{f} (u,z) -\D_{f}(y , z) - \Innp{\nabla f (y) - \nabla f(z), u-y }  +\innp{\nabla^{2} f (z) (y - z),u-z},
\end{align*}
	which, by the three point identity in  \cite[Proposition~2.3(ii)]{BBC2003}, yields \cref{theor:affine:character:PrightCf:innp} $\Rightarrow$ \cref{theor:affine:character:PrightCf:D}.  On the other hand, by   \cref{defn:BregmanDistance}, we know that \cref{theor:affine:character:PrightCf:D}
	$\Rightarrow$ \cref{theor:affine:character:PrightCf:=}. Altogether, we obtain that \cref{theor:affine:character:PrightCf:=} $\Leftrightarrow$ \cref{theor:affine:character:PrightCf:innp} $\Leftrightarrow$  \cref{theor:affine:character:PrightCf:D}.
	
	\cref{theor:affine:character:P:Lfor}: This follows easily from \cref{theor:affine:character:PrightCf:=} $\Leftrightarrow$ \cref{theor:affine:character:PrightCf:innp}.
\end{proof}

\subsection*{Generalizations of formulae of circumcenters}
The following result is a generalized version of \cite[Theorem~4.1]{BOyW2018}, which presents an explicit  formula  of the circumcenter  of finite sets.
\begin{theorem} \label{thm:LinIndpPformula}
	Let $z_{0}, z_{1}, \ldots, z_{m}$ be  affinely independent points in $\mathcal{H}$, and let $(\lambda_{1}, \ldots, \lambda_{m}) \in \mathbb{R}^{m}$.
Set $\I:=\{1, \ldots, m  \}$ and
	\begin{align*}
p:=	z_{0}+ (z_{1}-z_{0},\ldots,z_{m}-z_{0})
	G( z_{1}-z_{0},\ldots,z_{m}-z_{0})^{-1}
	\begin{pmatrix}
	\lambda_{1}-\innp{z_{0},z_{1}-z_{0}}\\
	\vdots\\
	\lambda_{m}-\innp{z_{0},z_{m}-z_{0}} \\
	\end{pmatrix},
	\end{align*}
	where $G( z_{1}-z_{0},\ldots,z_{m}-z_{0})$ is the
	Gram matrix defined as
	\begin{align*}
	G( z_{1}-z_{0},\ldots,  z_{m}-z_{0})
	:=
	\begin{pmatrix}
	\norm{z_{1}-z_{0}}^{2}   & \cdots & \innp{z_{1}-z_{0}, z_{m}-z_{0}}  \\
	\vdots & \cdots &   \vdots \\
	\innp{z_{m}-z_{0},z_{1}-z_{0}}   & \cdots & \norm{z_{m}-z_{0}}^{2} \\
	\end{pmatrix}.
	\end{align*}
	Then $\left\{ x \in \aff ( \{z_{0}, z_{1}, \ldots, z_{m} \} ) ~:~  (\forall i \in \I) ~ \innp{x , z_{i} -z_{0}} = \lambda_{i}  \right\}  = \{p\}$.
\end{theorem}

\begin{proof}
	Denote by $\Omega:=\left\{ x \in \aff ( \{z_{0}, z_{1}, \ldots, z_{m} \} ) ~:~  (\forall i \in \I) ~ \innp{x , z_{i} -z_{0}} = \lambda_{i}  \right\}  $.
	According to the assumption and \cite[Fact~2.8]{BOyW2018},
	$z_{1}-z_{0}, \ldots, z_{m}-z_{0} $ are linearly independent.
	Then by \cite[Fact~2.13]{BOyW2018}, the Gram matrix $G(z_{1}-z_{0},
	  \ldots, z_{m}-z_{0})$ is invertible. Set
	\begin{align*}
	\begin{pmatrix}
	\alpha_{1} \\
	\vdots\\
	\alpha_{m} \\
	\end{pmatrix}
	:= G(z_{1}-z_{0},  \ldots, z_{m}-z_{0})^{-1}
	\begin{pmatrix}
	\lambda_{1}-\innp{z_{0},z_{1}-z_{0}} \\
	\vdots\\
	\lambda_{m}-\innp{z_{0},z_{m}-z_{0}} \\
	\end{pmatrix}.
	\end{align*}
So, $p = z_{0}+\alpha_{1}(z_{1}-z_{0})+\alpha_{2}(z_{2}-z_{0})+\cdots+\alpha_{m}(z_{m}-z_{0}) \in \aff ( \{z_{0}, z_{1}, \ldots, z_{m} \} )$ is well-defined.

On the other hand,	using the definitions of $G(z_{1}-z_{0}, \ldots,
	z_{m}-z_{0})$, $(\alpha_{1},  \cdots, \alpha_{m})^{\intercal}$
	and $p$,  we see that
	\begin{align*}
	&G(z_{1}-z_{0}, \ldots,	z_{m}-z_{0})
	\begin{pmatrix}
	\alpha_{1} \\
	\vdots\\
	\alpha_{m} \\
	\end{pmatrix}
	=
	\begin{pmatrix}
	\lambda_{1}-\innp{z_{0},z_{1}-z_{0}} \\
	\vdots\\
	\lambda_{m}-\innp{z_{0},z_{m}-z_{0}} \\
	\end{pmatrix}\\
	\Leftrightarrow  &
	\begin{cases}
	\innp{ \alpha_{1}(z_{1}-z_{0})+ \cdots +\alpha_{m}(z_{m}-z_{0}), z_{1}-z_{0} } = \lambda_{1}-\innp{z_{0},z_{1}-z_{0}} \\
	~~~~~~~~~~\vdots \\
	\innp{\alpha_{1}(z_{1}-z_{0})+ \cdots +\alpha_{m}(z_{m}-z_{0}), z_{m}-z_{0} } = \lambda_{m}-\innp{z_{0},z_{m}-z_{0}}
	\end{cases}\\
	\Leftrightarrow &
	\begin{cases}
	\innp{p-z_{0},z_{1}-z_{0}} =\lambda_{1}-\innp{z_{0},z_{1}-z_{0}}\\
	~~~~~~~~~~\vdots \\
	\innp{p-z_{0}, z_{m}-z_{0}} =\lambda_{m}-\innp{z_{0},z_{m}-z_{0}}
	\end{cases}\\
	\Leftrightarrow & (\forall i \in \I) \quad \innp{p, z_{i}-z_{0}} = \lambda_{i}.
	\end{align*}
Hence, $p \in \Omega  $.
	
	In addition, assume $q \in \Omega$. Denote by $L:=\spn \{ z_{1} -z_{0}, \ldots, z_{m}-z_{0} \}$. Then $\{q, p\} \subseteq \Omega \subseteq \aff ( \{z_{0}, z_{1}, \ldots, z_{m} \} )$  implies that $q-p \in L$. 	
Moreover, 	$\{q, p\} \subseteq \Omega $ yields that
\begin{align*}
 (\forall i \in \I) \quad  \innp{ q , z_{i} -z_{0}} = \lambda_{i} \quad \text{and} \quad  \innp{ p , z_{i} -z_{0}} = \lambda_{i},
\end{align*}
which entails  $q-p \in L^{\perp}$. Therefore, $q-p \in L \cap L^{\perp} =\{0\}$, that is, $q=p$.

	Altogether,  the proof is complete.
\end{proof}

\section{Backward  Bregman circumcenters of finite sets} \label{sec:BackwardBregmancircumcenters}
In the whole section, suppose that $f \in \Gamma_{0} (\mathcal{H}) $  with $\inte \dom f \neq \varnothing$ and that $f$ is G\^ateaux differentiable  on $\inte \dom f$ and  that
\begin{empheq}[box=\mybluebox]{equation}   \label{eq:SBack}
\mathcal{S}:= \{q_{0}, q_{1}, \ldots, q_{m} \} \subseteq \inte \dom f \text{ and } \mathcal{S} \text{  is nonempty}.
\end{empheq}
Denote by $\I := \{1, \ldots, m \}$,
\begin{empheq}[box=\mybluebox]{equation}  \label{eq:Ef}
\overleftarrow{E}_{f}(\mathcal{S}):= \{ x \in \dom f ~:~  \D_{f} (x,q_{0}) =\D_{f} (x, q_{1}) =\cdots =\D_{f} (x, q_{m}) \},
\end{empheq}
\begin{align}  \label{eq:L}
L:= \aff ( \nabla f(\mathcal{S}) ) - \aff (  \nabla f(\mathcal{S}) ) = \spn \{ \nabla f(q_{1})-\nabla f(q_{0}), \ldots, \nabla f(q_{m})-\nabla f(q_{0}) \},
\end{align}
and
\begin{align} \label{eq:b}
(\forall i \in \I )\quad \beta_{i}:=  \innp{\nabla f(q_{i}), q_{i}} -f(q_{i})  - (\innp{\nabla f(q_{0}), q_{0}} - f(q_{0}) ).
\end{align}

\subsection*{Backward Bregman (pseudo-)circumcenter operators}
To introduce backward Bregman (pseudo-)circumcenter operators, we need the following lemmas.
\begin{lemma} \label{lemma:xinEp1pm}
We have $	x \in \overleftarrow{E}_{f}(\mathcal{S}) \Leftrightarrow  \left[  x \in \dom f \text{ and } (\forall i \in \I) \quad \innp{\nabla f(q_{i}) -\nabla f(q_{0}), x} = \beta_{i}\right]$.
\end{lemma}

\begin{proof}
Let $x \in \dom f$. According to \cref{defn:BregmanDistance},  for every $\{ i, j \} \subseteq  \{0,1, \ldots, m \}$,   we have that
	\begin{subequations}  \label{eq:lemma:xinEp1pm}
		\begin{align}
	 \D_{f} (x, q_{i}) =\D_{f} (x,q_{j})
	&	 \Leftrightarrow  f(x) -f(q_{i}) -\innp{\nabla f(q_{i}), x-q_{i}} = f(x) -f(q_{j}) -\innp{\nabla f(q_{j}), x-q_{j}} \\
		& \Leftrightarrow 	\innp{\nabla f(q_{j}) -\nabla f(q_{i}), x}  =\innp{\nabla f(q_{j}), q_{j}} - f(q_{j})   - (  \innp{\nabla f(q_{i}), q_{i}} -f(q_{i}) ).
		\end{align}
	\end{subequations}
 Hence,
 \begin{subequations}
 	\begin{align*}
  x \in  \overleftarrow{E}_{f}(\mathcal{S})
 &	\stackrel{\cref{eq:Ef}}{\Leftrightarrow} ~ (\forall i \in \I ) ~ \D_{f} (x, q_{i}) =\D_{f} (x,q_{0})  \\
 &	\stackrel{\cref{eq:lemma:xinEp1pm}}{\Leftrightarrow}   (\forall i \in  \I )  ~ \innp{\nabla f(q_{i}) -\nabla f(q_{0}), x} =\innp{\nabla f(q_{i}), q_{i}}-f(q_{i})   -(\innp{\nabla f(q_{0}), q_{0}} -f(q_{0})  )\\
& 	\stackrel{\cref{eq:b}}{\Leftrightarrow}   (\forall i \in \I ) ~ \innp{\nabla f(q_{i}) -\nabla f(q_{0}), x} =\beta_{i}.
 	\end{align*}
 \end{subequations}
\end{proof}

\begin{lemma} \label{lemma:CCSpseudoB}
	Suppose that $\nabla f(q_{0}), \nabla f(q_{1}), \ldots, \nabla f(q_{m})$ are affinely independent. Set
	\begin{align*}
p  := \nabla f(q_{0})+\alpha_{1} \left(\nabla f(q_{1})-\nabla f(q_{0}) \right)+ \cdots+ \alpha_{m} \left(\nabla f(q_{m})-\nabla f(q_{0})\right),
	\end{align*}
	where
	\begin{align*}
	\begin{pmatrix}
	\alpha_{1} \\
	\vdots\\
	\alpha_{m } \\
	\end{pmatrix}
	:=
	G( \nabla f(q_{1})-\nabla f(q_{0}),\ldots,\nabla f(q_{m})-\nabla f(q_{0}))^{-1}
	\begin{pmatrix}
	\beta_{1} -\innp{\nabla f(q_{0}),\nabla f(q_{1})-\nabla f(q_{0})}\\
	\vdots\\
	\beta_{m}-\innp{\nabla f(q_{0}),\nabla f(q_{m})-\nabla f(q_{0})} \\
	\end{pmatrix}.
	\end{align*}
	 Then $ \left\{ x \in \aff ( \{\nabla f(q_{0}), \nabla f(q_{1}), \ldots,\nabla f(q_{m}) \} ) ~:~  (\forall i \in \I) ~ \innp{x , \nabla f(q_{i}) -\nabla f(q_{0})} = \beta_{i}  \right\}  = \{p\}$.
\end{lemma}

\begin{proof}
	This follows immediately from \cref{thm:LinIndpPformula} by setting $z_{0}=\nabla f(q_{0})$, $z_{1}=\nabla f(q_{1})$, $\ldots$, $z_{m}=\nabla f(q_{m})$.
\end{proof}

\begin{lemma}  	\label{lemma:x-y:Lperp}
	Let $x $ and $y$ be in $\overleftarrow{E}_{f}(\mathcal{S}) $. Then $x-y \in L^{\perp}$.
	Consequently, $\left(\forall z \in \overleftarrow{E}_{f}(\mathcal{S}) \right)$ $\overleftarrow{E}_{f}(\mathcal{S}) =\dom f \cap \left( z+L^{\perp} \right)$.
\end{lemma}

\begin{proof}
Because $ \{x,y \} \subseteq \overleftarrow{E}_{f}(\mathcal{S}) \subseteq \dom f$, due to \cref{lemma:xinEp1pm}, for every $i \in \I$,
	\begin{subequations} \label{eq:x:y:affine}
		\begin{align}
		\innp{\nabla f(q_{i }) -\nabla f(q_{0}), x} = \beta_{i} \label{eq:x:y:affine:x},\\
		\innp{\nabla f(q_{i }) -\nabla f(q_{0}), y} =\beta_{i} \label{eq:x:y:affine:y}.
		\end{align}
	\end{subequations}
Then for every $i \in\I$, subtract \cref{eq:x:y:affine:y} from \cref{eq:x:y:affine:x} to observe that
	\begin{align*}
	(\forall i \in \I ) \quad  \innp{\nabla f(q_{i }) -\nabla f(q_{0}), x-y} =0,
	\end{align*}
	which, by \cref{eq:L}, yields that $x-y\in L^{\perp}$.
\end{proof}

\begin{lemma} \label{lemma:CCSsingletonOrempty}
	The set $ \aff ( \nabla f (\mathcal{S}) ) \cap \overleftarrow{E}_{f}(\mathcal{S})$ is either empty or a singleton.
\end{lemma}
\begin{proof}	
	Suppose that $ \aff ( \nabla f (\mathcal{S}) ) \cap \overleftarrow{E}_{f}(\mathcal{S}) \neq \varnothing$   and that  $\{ x_{1}, x_{2}\} \subseteq    \aff ( \nabla f (\mathcal{S}) ) \cap \overleftarrow{E}_{f}(\mathcal{S})$. 	
	Then
	\begin{align} \label{eq:theorem:CCSsingletonOrempty:L}
	x_{1}- x_{2}  \in \aff ( \nabla f(\mathcal{S}) ) - \aff (  \nabla f(\mathcal{S}) ) \stackrel{\cref{eq:L}}{=} L.
	\end{align}
	On the other hand,	 by  \cref{lemma:x-y:Lperp}, $x_{1}- x_{2}  \in L^{\perp}$. This combined with \cref{eq:theorem:CCSsingletonOrempty:L} to obtain that $x_{1}- x_{2}   \in L \cap L^{\perp} =\{0\}$.
	Therefore, the claimed result is true.
\end{proof}

We are now ready to define  backward Bregman circumcenters and pseudo-circumcenters.
\begin{definition} \label{defn:CCS:Bregman:left}
	Let $ \mathcal{P}( \inte \dom f)$ be  the set of all nonempty subsets of $\inte \dom f$ containing finitely many elements.
	For every $K \in \mathcal{P}( \inte \dom f)  $, set
	\begin{empheq}[box=\mybluebox]{equation*}
	\overleftarrow{E}_{f}(K):= \{ p \in \dom f ~:~ (\forall y \in K) ~ \D_{f} (p,y) \text{ is a singleton}\}.
		\end{empheq}
	\begin{enumerate}
		\item \label{defn:CCS:Bregman:left:} Define the \emph{backward Bregman circumcenter operator $\overleftarrow{\CCO{}}_{f}$ w.r.t.\,$f$} as
		\begin{empheq}[box=\mybluebox]{equation*}
		\overleftarrow{\CCO{}}_{f} : \mathcal{P}( \inte \dom f) \to 2^{\mathcal{H}} :  K \mapsto
		\aff (K )\cap \overleftarrow{E}_{f}( K ).
		\end{empheq}
	\item \label{defn:CCS:Bregman:left:ps}	Define the \emph{backward Bregman pseudo-circumcenter operator $\overleftarrow{\CCO{}}^{ps}_{f}$ w.r.t.\,$f$} as
	\begin{empheq}[box=\mybluebox]{equation*}
	\overleftarrow{\CCO{}}^{ps}_{f} :  \mathcal{P}( \inte \dom f) \to 2^{\mathcal{H}} :  K \mapsto  \aff ( \nabla f (K ))\cap \overleftarrow{E}_{f}( K ).
	\end{empheq}
	\end{enumerate}
	In particular, for every $K \in  \mathcal{P}( \inte \dom f)  $, we call the element  in  $	\overleftarrow{\CCO{}}_{f}(K) $ and   $\overleftarrow{\CCO{}}^{ps}_{f} (K)$ backward Bregman  circumcenter  and backward Bregman
pseudo-circumcenter of $K$, respectively,
\end{definition}
By \cref{lemma:CCSsingletonOrempty}, we know that $\left( \forall K \in  \mathcal{P}( \inte \dom f) \right)$ $\overleftarrow{\CCO{}}^{ps}_{f}(K)$ is either a singleton or an empty set.

\subsection*{Existence of backward Bregman  circumcenters}
\begin{proposition} \label{prop:formualCCS:Pleft:matrixEQ}
	 Set
	\begin{align*}
	&A:=
	\begin{pmatrix}
	\Innp{\nabla f (q_{1}) -\nabla f (q_{0}), q_{1} -q_{0} } & \ldots& \Innp{\nabla f (q_{1}) -\nabla f (q_{0}), q_{m} -q_{0} }\\
	\vdots & \ldots & \vdots\\
	\Innp{\nabla f (q_{m}) -\nabla f (q_{0}), q_{1} -q_{0} } & \ldots& \Innp{\nabla f (q_{m}) -\nabla f (q_{0}), q_{m} -q_{0} }\\
	\end{pmatrix}, \\
	&B:= \begin{pmatrix}
	\beta_{1} - \Innp{\nabla f (q_{1}) -\nabla f (q_{0}),q_{0} }\\
	\vdots\\
	\beta_{m} - \Innp{\nabla f (q_{m}) -\nabla f (q_{0}),q_{0} }
	\end{pmatrix} \text{ and}\\
	& \Lambda := \left\{  q_{0} +\sum^{m}_{i=1} \alpha_{i} \left( q_{i} -q_{0} \right) ~:~ (\alpha_{1}, \ldots, \alpha_{m} )^{\intercal}  \in \mathbb{R}^{m}  \text{ s.t. } A\alpha =B  \right\}.
	\end{align*}
	Then $\overleftarrow{\CCO{}}_{f}(\mathcal{S}) = \Lambda \cap \dom f$.
\end{proposition}

\begin{proof}
	According to \cref{defn:CCS:Bregman:left}\cref{defn:CCS:Bregman:left:} and \cref{lemma:xinEp1pm}, $p \in \overleftarrow{\CCO{}}_{f}(\mathcal{S}) $ if and only if there exists  $(\alpha_{1}, \ldots, \alpha_{m} )^{\intercal} \in \mathbb{R}^{m}$ such that $ p =q_{0} +\sum^{m}_{j=1} \alpha_{i} \left( q_{j} -q_{0} \right)  \in \dom f $ and
	\begin{align*}
   (\forall i \in \I) \quad  \Innp{\nabla f(q_{i}) -\nabla f(q_{0}), q_{0} +\sum^{m}_{j=1} \alpha_{j} \left( q_{j} -q_{0} \right) } = \beta_{i},
	\end{align*}
	which entails the desired result.
\end{proof}

\begin{theorem} \label{theorem:formualCCS:Pleft}
	Assume that $\overleftarrow{E}_{f}(\mathcal{S})  \neq \varnothing$.
Let $z \in \overleftarrow{E}_{f}(\mathcal{S})$.	Then  the following hold.
	\begin{enumerate}
		\item\label{theorem:formualCCS:Pleft:EQ} $\overleftarrow{\CCO{}}_{f}(\mathcal{S}) =\aff (\mathcal{S}) \cap \dom f \cap (z+L^{\perp})$.
		\item \label{theorem:formualCCS:Pleft:P} Suppose that $\mathcal{H} =\mathbb{R}^{n}$, that $f$ is Legendre such that $\dom f^{*}$ is open, that $f$	allows forward Bregman projections,   that $\aff (\mathcal{S}) \subseteq \inte \dom f$, and that $\nabla f (\aff (\mathcal{S})  ) $ is closed and convex such that $(\forall i \in \I)$ $\pm \left( \nabla f(q_{i }) -\nabla f(q_{0})  \right)\in \nabla f (\aff (\mathcal{S})  )-\overleftarrow{\Pro}^{f^{*}}_{\nabla f (\aff (\mathcal{S})  ) } (\nabla f (z)) $
		$($e.g.\,$ \nabla f (\aff (\mathcal{S})  )  $  is a closed affine subspace$)$.
		  Then $ \overrightarrow{\Pro}^{f}_{\aff( \mathcal{S}  )} (z) \in \overleftarrow{\CCO{}}_{f}(\mathcal{S})$.
	\end{enumerate}
\end{theorem}

\begin{proof}
	\cref{theorem:formualCCS:Pleft:EQ}: This is clear from  \cref{defn:CCS:Bregman:left}\cref{defn:CCS:Bregman:left:}  and \cref{lemma:x-y:Lperp}.

	\cref{theorem:formualCCS:Pleft:P}:
	Note that, by \cref{eq:SBack} and \cref{fact:charac:PrightCf}, $\varnothing \neq S \subseteq   \inte \dom f \cap \aff (\mathcal{S})$ and
$\overrightarrow{\Pro}^{f}_{\aff( \mathcal{S}  )} (z) \in \aff (\mathcal{S}) \subseteq \inte \dom f$. Hence, in view of \cref{theorem:formualCCS:Pleft:EQ} above and \cref{eq:L}, it remains to show that
	\begin{align} \label{eq:theorem:formualCCS:Pleft:P:betai}
	 (\forall i \in \I) \quad \Innp{\nabla f( q_{i }) -\nabla f(q_{0}), \overrightarrow{\Pro}^{f}_{\aff( \mathcal{S}  )} (z) -z} =0.
	\end{align}
Employ  \cref{fact:nablaf:nablaf*:id} and $\aff (\mathcal{S}) \subseteq \inte \dom f$ to deduce that $\nabla f \left( \aff (\mathcal{S}) \right) \subseteq \inte \dom f^{*}$.
Then apply \cref{fact:charac:PleftCf} with $f =f^{*}$, $C=\nabla f (\aff (\mathcal{S})  ) $, and $y=\nabla f (z)$ to obtain that $\overleftarrow{\Pro}^{f^{*}}_{\nabla f (\aff (\mathcal{S})  ) } (\nabla f (z))   \in \nabla f (\aff (\mathcal{S})  )  \cap \inte \dom f^{*}$
		and that
				\begin{align*}
			&	\Innp{\nabla f^{*}(\nabla f (z)) -\nabla f^{*} \left( \overleftarrow{\Pro}^{f^{*}}_{\nabla f (\aff (\mathcal{S})  ) } (\nabla f (z))  \right), \nabla f (\aff (\mathcal{S})  ) -\overleftarrow{\Pro}^{f^{*}}_{\nabla f (\aff (\mathcal{S})  ) } (\nabla f (z))  } \leq 0\\
			\Leftrightarrow &~
			\Innp{z -\nabla f^{*} \left( \overleftarrow{\Pro}^{f^{*}}_{\nabla f (\aff (\mathcal{S})  ) } (\nabla f (z))  \right), \nabla f (\aff (\mathcal{S})  ) -\overleftarrow{\Pro}^{f^{*}}_{\nabla f (\aff (\mathcal{S})  ) } (\nabla f (z))  } \leq 0 \quad (\text{by \cref{fact:nablaf:nablaf*:id}})\\
			\Leftrightarrow &~
				\Innp{z -\overrightarrow{\Pro}^{f}_{\aff( \mathcal{S}  )}(z), \nabla f (\aff (\mathcal{S})  ) -\overleftarrow{\Pro}^{f^{*}}_{\nabla f (\aff (\mathcal{S})  ) } (\nabla f (z))  } \leq 0, \quad (\text{by \cite[Poposition~7.1]{BWYY2009}})
			\end{align*}
which, combining with the assumption that $\pm \left( \nabla f(q_{i }) -\nabla f(q_{0})  \right)\in \nabla f (\aff (\mathcal{S})  )-\overleftarrow{\Pro}^{f^{*}}_{\nabla f (\aff (\mathcal{S})  ) } (\nabla f (z)) $, forces
\cref{eq:theorem:formualCCS:Pleft:P:betai}.
\end{proof}
Notice that the condition \enquote{$ \nabla f (\aff (\mathcal{S})  )  $  is a closed affine subspace} in \cref{theorem:formualCCS:Pleft}\cref{theorem:formualCCS:Pleft:P}  doesn't imply  $ \nabla f$ affine. Because in both our \cref{exam:Back:Rn}\cref{exam:Back:Rn:negative:}$\&$\cref{exam:Back:Rn:FD:} below $\mathcal{S}$ contains points in $\mathbb{R}^{n}$ with the first coordinates being the same and $\nabla f$ is  surjective, in both cases,  $ \nabla f (\aff (\mathcal{S})  )  $  is a closed affine subspace.

\begin{example}\label{exam:Back:Rn}
	Suppose that $\mathcal{H} =\mathbb{R}^{n}$.  Denote by  $(\forall x \in \mathcal{H} )$ $x= (x_{i})^{n}_{i=1} $.  Then the following statements hold.
	\begin{enumerate}
			
		\item \label{exam:Back:Rn:negative} Suppose that  $(\forall x \in \left[0, +\infty\right[^{n} )$ $f(x) =\sum^{n}_{i=1} x_{i} \ln( x_{i}) -x_{i}$.
		\begin{enumerate}
				\item \label{exam:Back:Rn:negative:xy}  Suppose that $\mathcal{S}:= \{x,y \} \subseteq \left]0, +\infty\right[^{n}$ such that $x$ and $y$  are pairwise distinct. Let $p \in \mathcal{H} $.  Then $ p \in 	\overleftarrow{\CCO{}}_{f}(\mathcal{S}) $ if and only if $p=\alpha x +  (1-\alpha) y   \in \left[0, +\infty\right[^{n}$,   where $\alpha:= \frac{\sum^{n}_{i=1} y_{i}-x_{i} -y_{i} \ln \left(\frac{y_{i}}{x_{i}} \right) }{\sum^{n}_{i=1}(x_{i}-y_{i}) \ln \left(\frac{y_{i}}{x_{i}} \right) }$.
				
			\item \label{exam:Back:Rn:negative:xyz} Suppose that $\mathcal{S}:= \{x,y,z\} \subseteq \left]0, +\infty\right[^{n}$ such that $x$,  $y$, and $ z$ are pairwise distinct. Let $p \in \mathcal{H} $.  Then $ p \in 	\overleftarrow{\CCO{}}_{f}(\mathcal{S}) $ if and only if $p=x + \alpha (y-x) +\beta (z-x) \in \left[0, +\infty\right[^{n}$, where $(\alpha, \beta)^{\intercal} \in \mathbb{R}^{2}$ solves the following equation
			\begin{align*}
			\begin{pmatrix}
			\sum^{n}_{i=1}(y_{i} -x_{i}) \ln(\frac{y_{i}}{x_{i}}) & 	\sum^{n}_{i=1}(z_{i} -x_{i}) \ln(\frac{y_{i}}{x_{i}})\\
				\sum^{n}_{i=1}(y_{i} -x_{i}) \ln(\frac{z_{i}}{x_{i}}) & 	\sum^{n}_{i=1}(z_{i} -x_{i}) \ln(\frac{z_{i}}{x_{i}})
			\end{pmatrix}
			\begin{pmatrix}
			\alpha\\
			\beta
			\end{pmatrix}
			=
			\begin{pmatrix}
			\sum^{n}_{i=1}y_{i} -x_{i} -x_{i} \ln(\frac{y_{i}}{x_{i}})\\
			\sum^{n}_{i=1} z_{i} -x_{i} -x_{i} \ln(\frac{z_{i}}{x_{i}})
			\end{pmatrix}.
			\end{align*}
			\item \label{exam:Back:Rn:negative:} Suppose that $\mathcal{H}=\mathbb{R}^{3}$ and that  $\mathcal{S}:= \{(1,1,1), (1,2,1), (1,1,2)\} $. Then
			\begin{align*}
				\overleftarrow{\CCO{}}_{f}(\mathcal{S}) = \left\{ (1,1,1) + \frac{1-\ln 2}{\ln 2} (0,1,0) +\frac{1-\ln 2}{\ln 2}  (0,0,1)   \right\}.
			\end{align*}
		\end{enumerate}
	\item \label{exam:Back:Rn:FD}	Suppose that $(\forall x \in \left[0, 1\right]^{n} )$ $f(x) =\sum^{n}_{i=1} x_{i} \ln( x_{i}) + (1-x_{i}) \ln (1-x_{i})$.
	\begin{enumerate}
	\item  \label{exam:Back:Rn:FD:xyz} Suppose that $\mathcal{S}:= \{x,y,z\} \subseteq \left]0, 1\right[^{n} $ such that $x$,  $y$, and $ z$ are pairwise distinct. Let $p \in \mathcal{H} $.  Then $ p \in 	\overleftarrow{\CCO{}}_{f}(\mathcal{S}) $ if and only if $p=x + \alpha (y-x) +\beta (z-x) \in \left[ 0, 1\right]^{n} $, where $(\alpha, \beta)^{\intercal}$ solves the following equation
	\begin{align*}
	\begin{pmatrix}
 \sum^{n}_{i=1} (y_{i} -x_{i} )  \ln \left(  \frac{x_{i}}{y_{i}} \frac{1-y_{i}}{1-x_{i}}  \right)  & 	\sum^{n}_{i=1} (z_{i} -x_{i} )  \ln \left(  \frac{x_{i}}{y_{i}} \frac{1-y_{i}}{1-x_{i}}  \right)\\
\sum^{n}_{i=1} (y_{i} -x_{i} )  \ln \left(  \frac{x_{i}}{z_{i}} \frac{1-z_{i}}{1-x_{i}}  \right) & 	\sum^{n}_{i=1} (z_{i} -x_{i} )  \ln \left(  \frac{x_{i}}{z_{i}} \frac{1-z_{i}}{1-x_{i}}  \right)
	\end{pmatrix}
	\begin{pmatrix}
	\alpha\\
	\beta
	\end{pmatrix}
	=
	\begin{pmatrix}
	\sum^{n}_{i=1} \ln \left(    \frac{1-y_{i}}{1-x_{i}}   \right) -x_{i}  \ln \left(  \frac{x_{i}}{y_{i}} \frac{1-y_{i}}{1-x_{i}}  \right)\\
	\sum^{n}_{i=1} \ln \left(    \frac{1-z_{i}}{1-x_{i}}   \right) -x_{i}  \ln \left(  \frac{x_{i}}{z_{i}} \frac{1-z_{i}}{1-x_{i}}  \right)
	\end{pmatrix}.
	\end{align*}
	\item  \label{exam:Back:Rn:FD:} Suppose that $\mathcal{H}=\mathbb{R}^{3}$ and that $\mathcal{S}:= \{(\frac{1}{4},\frac{1}{4},\frac{1}{4}), (\frac{1}{4},\frac{1}{2},\frac{1}{4}), (\frac{1}{4},\frac{1}{4},\frac{1}{2})\} $. Then
	\begin{align*}
	\overleftarrow{\CCO{}}_{f}(\mathcal{S}) = \left\{ \left(\frac{1}{4},\frac{1}{4},\frac{1}{4}\right) +\frac{3\ln 3 -4 \ln 2 }{\ln 3} \left(0,\frac{1}{4},0\right) +\frac{3\ln 3 -4 \ln 2 }{\ln 3}  \left(0,0,\frac{1}{4}\right)   \right\}.
	\end{align*}
\end{enumerate}	
	\end{enumerate}
\end{example}

\begin{proof}
	\cref{exam:Back:Rn:negative}:
	According to \cite[Proposition~3.5]{BB1997Legendre} and \cref{defn:BregmanDistance},
		\begin{align}\label{eq:exam:back:Rn}
 	\left(\forall \{u,v\} \subseteq \left[0, +\infty\right[^{n}  \right)
	\quad D_{f} (u,v) =\sum^{n}_{i=1}   u_{i} (\ln (u_{i}) -\ln (v_{i})) +v_{i} -u_{i} .
	\end{align}
	
	\cref{exam:Back:Rn:negative:xy}: This is clear from  \cref{eq:exam:back:Rn} and \cref{defn:CCS:Bregman:left}\cref{defn:CCS:Bregman:left:}.
	
	\cref{exam:Back:Rn:negative:xyz}: In view of \cref{defn:CCS:Bregman:left}\cref{defn:CCS:Bregman:left:},  $ p \in 	\overleftarrow{\CCO{}}_{f}(\mathcal{S}) $ if and only if $p=x + \alpha (y-x) +\beta (z-x) \in \left[0, +\infty\right[^{n}$ such that
	\begin{align*}
&	\D_{f} (p, x) =\D_{f} (p, y)=\D_{f} (p, z) \\
	\Leftrightarrow &
	\begin{cases}
	\sum^{n}_{i=1} p_{i} \ln (\frac{y_{i}}{x_{i}}) = \sum^{n}_{i=1}(y_{i} -x_{i})\\
	\sum^{n}_{i=1} p_{i} \ln (\frac{z_{i}}{x_{i}}) = \sum^{n}_{i=1}(z_{i} -x_{i})
	\end{cases}\\
	 \Leftrightarrow &
	\begin{cases}
	\alpha\sum^{n}_{i=1}(y_{i} -x_{i}) \ln(\frac{y_{i}}{x_{i}})  +\beta 	\sum^{n}_{i=1}(z_{i} -x_{i}) \ln(\frac{y_{i}}{x_{i}}) = \sum^{n}_{i=1}\left(y_{i} -x_{i} -x_{i} \ln(\frac{y_{i}}{x_{i}}) \right)\\
	\alpha   \sum^{n}_{i=1}(y_{i} -x_{i}) \ln(\frac{z_{i}}{x_{i}})   +   \beta 	\sum^{n}_{i=1}(z_{i} -x_{i}) \ln(\frac{z_{i}}{x_{i}})   = 	\sum^{n}_{i=1}\left(z_{i} -x_{i} -x_{i} \ln(\frac{z_{i}}{x_{i}}) \right)
	\end{cases},	
	\end{align*}
	 which deduces the required result.
	
	\cref{exam:Back:Rn:negative:}: The desired result follows clearly from \cref{exam:Back:Rn:negative:xyz} above.
	
		\cref{exam:Back:Rn:FD:xyz}: Employing \cref{defn:CCS:Bregman:left}\cref{defn:CCS:Bregman:left:}, \cite[Proposition~3.5]{BB1997Legendre} and \cref{eq:exam:back:R:negative:xyz:D}, we observe that in this case, $ p \in 	\overleftarrow{\CCO{}}_{f}(\mathcal{S}) $ if and only if $p=x + \alpha (y-x) +\beta (z-x) \in \left[0, 1\right]^{n}$ such that
	\begin{align*}
	&D_{f} (p,x) =D_{f} (p,y)=D_{f} (p,z) \\
	\Leftrightarrow &\begin{cases}
		\sum^{n}_{i=1} p_{i}  \ln (\frac{p_{i}}{x_{i}})  + (1 -p_{i})\ln (\frac{1-p_{i}}{1-x_{i}})  = \sum^{n}_{i=1} p_{i}  \ln (\frac{p_{i}}{y_{i}})  + (1 -p_{i})\ln (\frac{1-p_{i}}{1-y_{i}}) \\
			\sum^{n}_{i=1} p_{i}  \ln (\frac{p_{i}}{x_{i}})  + (1 -p_{i})\ln (\frac{1-p_{i}}{1-x_{i}})  = \sum^{n}_{i=1} p_{i}  \ln (\frac{p_{i}}{z_{i}})  + (1 -p_{i})\ln (\frac{1-p_{i}}{1-z_{i}})
	\end{cases}\\
	\Leftrightarrow &\begin{cases}
	\sum^{n}_{i=1} p_{i}  \ln \left(  \frac{x_{i}}{y_{i}} \frac{1-y_{i}}{1-x_{i}}  \right) = \sum^{n}_{i=1} \ln \left(    \frac{1-y_{i}}{1-x_{i}}  \right) \\
		\sum^{n}_{i=1} p_{i}  \ln \left(  \frac{x_{i}}{z_{i}} \frac{1-z_{i}}{1-x_{i}}  \right) = \sum^{n}_{i=1} \ln \left(    \frac{1-z_{i}}{1-x_{i}}  \right)
		\end{cases}\\
		\Leftrightarrow &\begin{cases}
		\alpha \sum^{n}_{i=1} (y_{i} -x_{i} )  \ln \left(  \frac{x_{i}}{y_{i}} \frac{1-y_{i}}{1-x_{i}}  \right)  +
		\beta \sum^{n}_{i=1} (z_{i} -x_{i} )  \ln \left(  \frac{x_{i}}{y_{i}} \frac{1-y_{i}}{1-x_{i}}  \right)
		 = \sum^{n}_{i=1} \ln \left(    \frac{1-y_{i}}{1-x_{i}}   \right) -x_{i}  \ln \left(  \frac{x_{i}}{y_{i}} \frac{1-y_{i}}{1-x_{i}}  \right)\\
		 	\alpha \sum^{n}_{i=1} (y_{i} -x_{i} )  \ln \left(  \frac{x_{i}}{z_{i}} \frac{1-z_{i}}{1-x_{i}}  \right)  +
		 \beta \sum^{n}_{i=1} (z_{i} -x_{i} )  \ln \left(  \frac{x_{i}}{z_{i}} \frac{1-z_{i}}{1-x_{i}}  \right)
		 = \sum^{n}_{i=1} \ln \left(    \frac{1-z_{i}}{1-x_{i}}   \right) -x_{i}  \ln \left(  \frac{x_{i}}{z_{i}} \frac{1-z_{i}}{1-x_{i}}  \right)
			\end{cases},
	\end{align*}
	which implies the required result.

	\cref{exam:Back:Rn:FD:}: This follows from 	\cref{exam:Back:Rn:FD:xyz} and some easy algebra.
\end{proof}

\begin{example} \label{exam:back:R}
	Suppose that $\mathcal{H}=\mathbb{R}$. Then the following statements hold.
	\begin{enumerate}
		\item  \label{exam:back:R:negative} Suppose that  $(\forall x \in \left[0, +\infty\right[ )$ $f(x) =x \ln( x) -x$, and that $\mathcal{S}:= \{x,y,z\} \subseteq \left]0, +\infty\right[$ such that $x$,  $y$, and $ z$ are pairwise distinct. Then $\overleftarrow{\CCO{}}_{f}(\mathcal{S}) = \varnothing$.

		\item  \label{exam:back:R:FD} Suppose that $(\forall x \in \left[0, 1\right] )$ $f(x) =x \ln( x) + (1-x) \ln (1-x)$, and that $\mathcal{S}:= \{x,y,z\} \subseteq  \left]0, 1\right[$ such that $x$,  $y$, and $ z$ are pairwise distinct. Then $\overleftarrow{\CCO{}}_{f}(\mathcal{S}) = \varnothing$.

	\end{enumerate}
	
\end{example}

\begin{proof}
	\cref{exam:back:R:negative}:
	Let $p \in \left[0, +\infty\right[ $. According to \cref{eq:exam:back:Rn} with $n=1$,  it is easy to see that
	\begin{align} \label{eq:exam:back:R:negative:xyz}
	D_{f} (p,x) =D_{f} (p,y)=D_{f} (p,z) \Leftrightarrow p =\frac{y-x}{\ln (y) -\ln (x)}  =\frac{z-x}{\ln (z) -\ln (x)}.
	\end{align}
	Set $g:\left]0, +\infty\right[ \to \mathbb{R} : t \mapsto \frac{t-x}{\ln(t) -\ln(x)} $. Then   clearly, $\left(\forall t \in \left]0, +\infty\right[  \smallsetminus \{x\} \right) $ $g(t) \in \mathbb{R}_{++}$,
	\begin{align*}
	\left(\forall t \in \left]0, +\infty\right[  \smallsetminus \{x\} \right) \quad g'(t) =\frac{(\frac{x}{t} -\ln \frac{x}{t}) -1}{( \ln (t) -\ln (x))^{2} } >0,
	\end{align*}
	which implies that $g$ is strictly increasing. This combined with the assumption, $x$,  $y$, and $ z$ are pairwise distinct, deduces that there is no $p \in \left[0, +\infty\right[$ satisfying \cref{eq:exam:back:R:negative:xyz}. Therefore, by \cref{defn:CCS:Bregman:left}\cref{defn:CCS:Bregman:left:}, $\overleftarrow{\CCO{}}_{f}(\mathcal{S}) = \varnothing$.
	
	\cref{exam:back:R:FD}: Now, in view of \cref{defn:BregmanDistance},
	\begin{align} \label{eq:exam:back:R:negative:xyz:D}
	(\forall \{u,v\} \subseteq \left]0, 1\right[ ) \quad D_{f} (u,v) =u  \ln (\frac{u}{v})  + (1 -u)\ln (\frac{1-u}{1-v}),
	\end{align}
	and
	\begin{subequations}\label{eq:exam:back:R:FD}
		\begin{align}
		&D_{f} (p,x) =D_{f} (p,y)=D_{f} (p,z) \\
		\Leftrightarrow &p = \left(\ln \frac{1-y}{1-x} \right) \left( \ln \frac{1-y}{1-x} -\ln \frac{y}{x} \right)^{-1}  =\left(\ln \frac{1-z}{1-x} \right) \left( \ln \frac{1-z}{1-x} -\ln \frac{z}{x} \right)^{-1}.
		\end{align}
	\end{subequations}
	Set $g:\left]0, 1\right[  \to \mathbb{R} : t \mapsto \left(\ln \frac{1-t}{1-x} \right) \left( \ln \frac{1-t}{1-x} -\ln \frac{t}{x} \right)^{-1} $. Then  some  algebra deduces that $(\forall t \in\left]0, 1\right[ \smallsetminus \{x\}  )$ $g(t) \in \left[ 0,1 \right]$ and $g'(t) >0$, which, combining with the assumption, $x$,  $y$, and $ z$ are pairwise distinct, implies that  there exists no $p \in \left[0, 1\right] $ satisfying \cref{eq:exam:back:R:FD}. Therefore, by \cref{defn:CCS:Bregman:left}\cref{defn:CCS:Bregman:left:}, $\overleftarrow{\CCO{}}_{f}(\mathcal{S}) = \varnothing$.
\end{proof}

\subsection*{Explicit  formula of  backward Bregman pseudo-circumcenters}

\begin{proposition} \label{prop:formualCCS:matrixEQ}
	Set
	\begin{align*}
	&A:=
	\begin{pmatrix}
	\Innp{\nabla f (q_{1}) -\nabla f (q_{0}), \nabla f (q_{1}) -\nabla f (q_{0}) } & \ldots& \Innp{\nabla f (q_{1}) -\nabla f (q_{0}), \nabla f (q_{m}) -\nabla f (q_{0}) }\\
	\vdots & \ldots & \vdots\\
	\Innp{\nabla f (q_{m}) -\nabla f (q_{0}), \nabla f (q_{1}) -\nabla f (q_{0})} & \ldots& \Innp{\nabla f (q_{m}) -\nabla f (q_{0}), \nabla f (q_{m}) -\nabla f (q_{0}) }\\
	\end{pmatrix},\\
	&B:= \begin{pmatrix}
	\beta_{1} - \Innp{\nabla f (q_{1}) -\nabla f (q_{0}), \nabla f (q_{0}) }\\
	\vdots\\
	\beta_{m} - \Innp{\nabla f (q_{m}) -\nabla f (q_{0}), \nabla f (q_{0}) }
	\end{pmatrix}
	\text{ and } \\
	& \Omega := \left\{   \nabla f (q_{0}) +\sum^{m}_{i=1} \alpha_{i} \left( \nabla f (q_{i}) -\nabla f (q_{0}),\right)   ~:~ (\alpha_{1}, \ldots, \alpha_{m} )^{\intercal}  \in \mathbb{R}^{m} \text{ s.t. } A\alpha =B    \right\}.
	\end{align*}
	Then $  \overleftarrow{\CCO{}}^{ps}_{f}(\mathcal{S})  = \Omega \cap \dom f$.
\end{proposition}

\begin{proof}
	The proof is similar to that of \cref{prop:formualCCS:Pleft:matrixEQ}, but this time we exploit
 \cref{defn:CCS:Bregman:left}\cref{defn:CCS:Bregman:left:ps}  and \cref{lemma:xinEp1pm}.
\end{proof}

\begin{example}\label{exam:Backpseudo:Rn}
	Suppose that $\mathcal{H} =\mathbb{R}^{n}$.  Denote by  $(\forall x \in \mathcal{H} )$ $x= (x_{i})^{n}_{i=1} $.
	Suppose that  $(\forall x \in \left[0, +\infty\right[^{n} )$ $f(x) =\sum^{n}_{i=1} x_{i} \ln( x_{i}) -x_{i}$.
	and  that $\mathcal{S}:= \{x,y \} \subseteq \left]0, +\infty\right[^{n}$ such that $x$ and $y$  are pairwise distinct. Let $p \in \mathcal{H} $.  Then $ p \in 	 \overleftarrow{\CCO{}}^{ps}_{f}(\mathcal{S}) $ if and only if $p=\Big( \alpha \ln (x_{1}) +  (1-\alpha) \ln (y_{1}) , \ldots  ,\alpha \ln (x_{n}) +  (1-\alpha) \ln (y_{n}) \Big)^{\intercal}  \in \left[0, +\infty\right[^{n}$,   where $\alpha:= \frac{ \sum^{n}_{i=1} x_{i}-y_{i} + \ln (y_{i}) \ln \left(\frac{y_{i}}{x_{i}} \right) }{\sum^{n}_{i=1}  \left(\ln \frac{y_{i}}{x_{i}} \right)^{2 }}$.
\end{example}
\begin{proof}
	This is clear  by substituting $m=1$, $q_{1}=x$ and $q_{0} =y$ in \cref{prop:formualCCS:matrixEQ}.
\end{proof}

\begin{theorem} \label{theorem:formualCCS}
	Suppose that $\overleftarrow{E}_{f}(\mathcal{S}) \neq \varnothing$.  Let $x \in \overleftarrow{E}_{f}(\mathcal{S}) $. Then the following statements hold.
	\begin{enumerate}
		\item \label{theorem:formualCCS:id}  $ \overleftarrow{\CCO{}}^{ps}_{f}(\mathcal{S})  = \aff ( \nabla f (\mathcal{S}) ) \cap \overleftarrow{E}_{f}(\mathcal{S})  = \aff ( \nabla f (\mathcal{S}) ) \cap \dom f \cap \left( x+L^{\perp} \right)$.

		 \item \label{theorem:formualCCS:EucP}  Suppose that  $ \Pro_{\aff(\nabla f( \mathcal{S} ) )}  (x ) \in \dom f$ $($e.g., $\aff(\nabla f( \mathcal{S} ) ) \subseteq \dom f$ or $\dom f =\mathcal{H}$$)$.  Then  $\overleftarrow{\CCO{}}^{ps}_{f}(\mathcal{S})=\Pro_{\aff(\nabla f( \mathcal{S} ) )} (x)  $.
	\end{enumerate}
\end{theorem}

\begin{proof}
	\cref{theorem:formualCCS:id}: This is a direct result of \cref{defn:CCS:Bregman:left}\cref{defn:CCS:Bregman:left:ps}  and \cref{lemma:x-y:Lperp}.

	\cref{theorem:formualCCS:EucP}: Notice that $\Pro_{\aff(\nabla f( \mathcal{S} ) )} (x) \in \aff(\nabla f( \mathcal{S} ) ) \cap   \dom f$.
Furthermore,
	\begin{align*}
	x - \Pro_{\aff(\nabla f( \mathcal{S} ) )} (x) & = x - \Pro_{\nabla f(q_{0}) +L} (x) \quad (\text{by \cref{eq:L}})\\
	&= x-\nabla f(q_{0})  -  \Pro_{ L} (x-\nabla f(q_{0}) ) \quad (\text{by \cite[Proposition~3.19]{BC2017}})\\
	&= \Pro_{ L^{\perp}} (x-\nabla f(q_{0})  ) \in L^{\perp} .  \quad (\text{by  \cite[Theorem~5.8]{D2012}})
	\end{align*}
Altogether,  $\Pro_{\aff(\nabla f( \mathcal{S} ) )} (x) \in \aff ( \nabla f (\mathcal{S}) ) \cap \dom f \cap \left( x+L^{\perp} \right)$, which combining with  \cref{theorem:formualCCS:id} above and \cref{lemma:CCSsingletonOrempty} to deduce the required result.
\end{proof}

\begin{corollary} \label{cor:equi:Ef:CCfS}
	Suppose that $ \Pro_{\aff(\nabla f( \mathcal{S} ) )}  (\overleftarrow{E}_{f}(\mathcal{S}) ) \subseteq \dom f$ $($e.g., $\dom f =\mathcal{H}$$)$.  Then $\overleftarrow{E}_{f}(\mathcal{S})  \neq \varnothing$ if and only if $\overleftarrow{\CCO{}}^{ps}_{f}(\mathcal{S}) \neq \varnothing$.
\end{corollary}

\begin{proof}
	It is easy to verify the equivalence by \cref{defn:CCS:Bregman:left}\cref{defn:CCS:Bregman:left:ps} and \cref{theorem:formualCCS}\cref{theorem:formualCCS:EucP}.
\end{proof}

The following result provides an explicit formula for backward Bregman pseudo-circumcenters.
\begin{theorem} \label{thm:unique:LinIndpPformula}
Suppose that $\nabla f(q_{0}), \nabla f(q_{1}), \ldots, \nabla f(q_{m})$ are affinely independent and that $\aff ( \nabla f (\mathcal{S}) ) \subseteq \dom f$.
	Then $\overleftarrow{\CCO{}}^{ps}_{f}(\mathcal{S}) \neq \varnothing$.
	Moreover,
	\begin{align} \label{eq:thm:unique:LinIndpPformula}
	\overleftarrow{\CCO{}}^{ps}_{f}(\mathcal{S})  = \nabla f(q_{0})+\alpha_{1} \left(\nabla f(q_{1})-\nabla f(q_{0}) \right)+ \cdots+ \alpha_{m} \left(\nabla f(q_{m})-\nabla f(q_{0})\right),
		\end{align}
		where
		\begin{align*}
		\begin{pmatrix}
		\alpha_{1} \\
		\vdots\\
		\alpha_{m } \\
		\end{pmatrix}
		:=
	G( \nabla f(q_{1})-\nabla f(q_{0}),\ldots,\nabla f(q_{m})-\nabla f(q_{0}))^{-1}
	\begin{pmatrix}
	\beta_{1} -\innp{\nabla f(q_{0}),\nabla f(q_{1})-\nabla f(q_{0})}\\
	\vdots\\
	\beta_{m}-\innp{\nabla f(q_{0}),\nabla f(q_{m})-\nabla f(q_{0})} \\
	\end{pmatrix}.
	\end{align*}
\end{theorem}

\begin{proof}
	According to the assumption and \cite[Fact~2.8]{BOyW2018},
	$\nabla f(q_{1})-\nabla f(q_{0}), \ldots, \nabla f(q_{m})-\nabla f(q_{0}) $ are linearly independent.
	Then by \cite[Fact~2.13]{BOyW2018}, the Gram matrix $G\left(\nabla f(q_{1})-\nabla f(q_{0}), \ldots, \nabla f(q_{m})-\nabla f(q_{0}) \right)$ is invertible.
	Therefore, the required result follows immediately from  \cref{prop:formualCCS:matrixEQ}.	
\end{proof}

\section{Forward Bregman circumcenters of finite sets}\label{sec:ForwardBregmancircumcenters}
Throughout this section,  suppose that $f \in \Gamma_{0} (\mathcal{H}) $  with $\inte \dom f \neq \varnothing$ and that $f$ is G\^ateaux  differentiable  on $\inte \dom f$, and  that
\begin{empheq}[box=\mybluebox]{equation}   \label{eq:Sfor}
\mathcal{S}:= \{p_{0}, p_{1}, \ldots, p_{m} \} \subseteq  \dom f \text{ and } \mathcal{S} \text{ is nonempty}.
\end{empheq}
Denote by  $\I := \{1, \ldots, m\}$,
\begin{empheq}[box=\mybluebox]{equation}  \label{eq:Eright}
\overrightarrow{E}_{f}(\mathcal{S} ):= \{ y \in \inte \dom f~:~  \D_{f} (p_{0},y) =\D_{f} (p_{1},y) =\cdots =\D_{f} (p_{m},y) \},
\end{empheq}
\begin{align}  \label{eq:M}
M :=\aff ( \mathcal{S} ) - \aff ( \mathcal{S} ) = \spn \{ p_{1} -p_{0}, \ldots, p_{m} -p_{0} \},
\end{align}
and
\begin{align} \label{eq:eta}
(\forall i \in \I )\quad \eta_{i}:= f(p_{i}) -f (p_{0}).
\end{align}

\subsection*{Forward Bregman (pseudo-)circumcenter operators}
The following lemmas are helpful in this section.
\begin{lemma} \label{lemma:xinEp1pm:right}
We have $	y \in  \overrightarrow{E}_{f}(\mathcal{S} ) \Leftrightarrow  \left[ y \in \inte \dom f \text{ and } (\forall i \in \I) \quad \innp{\nabla f(y) , p_{i }-p_{0}} = \eta_{i}  \right]$ .
\end{lemma}

\begin{proof}
Let $y \in \inte \dom f$.	Then  $f(y) \in \mathbb{R}$. Hence, for every $\{ i, j  \} \subseteq   \I$, by \cref{defn:BregmanDistance},
	 \begin{subequations}\label{eq:DfpixDfpjx}
	 	\begin{align}
	 	\D_{f} (p_{i}, y) =\D_{f} ( p_{j}, y) & \Leftrightarrow f(p_{i}) - f(y) -\innp{\nabla f(y), p_{i} -y} = f(p_{j}) - f(y) -\innp{\nabla f(y), p_{j} -y} \\
	 	& \Leftrightarrow \innp{\nabla f(y) ,  p_{j} - p_{i} }  = f(p_{j})  -f(p_{i}).
	 	\end{align}
	 \end{subequations}
Therefore,
		\begin{align*}
		  y\in \overrightarrow{E}_{f}(\mathcal{S} )
		\stackrel{\cref{eq:Eright}}{\Leftrightarrow} & (\forall i \in \I) \quad  \D_{f} (p_{i},y) =\D_{f} (p_{0},y ) \\
		\stackrel{\cref{eq:DfpixDfpjx}}{\Leftrightarrow} & (\forall i \in\I) \quad \innp{\nabla f(y) ,  p_{i}  -p_{0} } = f(p_{i})  -f(p_{0})\\
		\stackrel{\cref{eq:eta}}{\Leftrightarrow} & (\forall i \in \I) \quad \innp{\nabla f(y) ,  p_{i}  -p_{1} } =\eta_{i}.
		\end{align*}
\end{proof}

\begin{lemma} \label{lemma:CCSpsF}
	Suppose that  $p_{0}, p_{1}, \ldots, p_{m}$ are affinely independent. Set
	\begin{align*}
q: =  p_{0}+ (p_{1}-p_{0},\ldots,p_{m}-p_{0})
G( p_{1}-p_{0},\ldots,p_{m}-p_{0})^{-1}
\begin{pmatrix}
\eta_{1}-\innp{p_{0},p_{1}-p_{0}}\\
\vdots\\
\eta_{m}-\innp{p_{0},p_{m}-p_{0}} \\
\end{pmatrix},
	\end{align*}
		Then $	\left\{ x \in \aff ( \{p_{0}, p_{1}, \ldots, p_{m} \} ) ~:~  (\forall i \in \I) ~ \innp{x , p_{i} -p_{0}} = \eta_{i}  \right\}  = \{q\}$.
\end{lemma}

\begin{proof}
	Apply \cref{thm:LinIndpPformula} with $z_{0}=p_{0}$ and $(\forall i \in \I)$ $z_{i}=p_{i}$ and $\lambda_{i}=\eta_{i}$ to deduce the desired result.
\end{proof}

\begin{lemma} \label{lemma:nablafyi:in:Mperp}
	Let $x$ and $y$ be in $ \overrightarrow{E}_{f}(\mathcal{S} )$. Then $\nabla f (y) -\nabla f (x) \in M^{\perp}$.
\end{lemma}

\begin{proof}
Notice that	$\{ x,y\} \subseteq \overrightarrow{E}_{f}(\mathcal{S})$ implies that $\{x,y\} \subseteq \inte \dom f$, and  that, by \cref{lemma:xinEp1pm:right},
	\begin{subequations}
		\begin{align}
	(\forall i \in \I) \quad 	\innp{\nabla f(x) , p_{i }-p_{0}} = \eta_{i}, \label{lemma:nablafyi:in:Mperp:y1}\\
	(\forall i \in \I) \quad 	\innp{\nabla f(y) , p_{i}-p_{0}} = \eta_{i}. \label{lemma:nablafyi:in:Mperp:y2}
		\end{align}
	\end{subequations}
Subtracting \cref{lemma:nablafyi:in:Mperp:y1} from \cref{lemma:nablafyi:in:Mperp:y2} to deduce that  $	(\forall i \in \I) $ $ \innp{\nabla f(y ) -\nabla f(x), p_{i }-p_{0}} = 0$,
which, connecting with the definition of $M$ in \cref{eq:M}, yields the desired result.
\end{proof}

\begin{lemma} \label{lem:CCOfrS}
	Suppose that $f$ is Legendre. Then
	$\left( \nabla f^{*} \left( \aff    \mathcal{S} \right)   \right) \cap  \overrightarrow{E}_{f}(\mathcal{S}) $  is either an empty set or a singleton.
\end{lemma}

\begin{proof}
	Suppose that $ \left( \nabla f^{*} \left( \aff    \mathcal{S} \right)   \right) \cap  \overrightarrow{E}_{f}(\mathcal{S})  \neq \varnothing$ and that  $\{y_{1}, y_{2} \}  \subseteq \left( \nabla f^{*} \left( \aff    \mathcal{S} \right)   \right) \cap  \overrightarrow{E}_{f}(\mathcal{S}) $.
	Then, by \cref{fact:nablaf:nablaf*:id} and \cref{eq:M},
	\begin{align} \label{eq:lem:CCOfrS:M}
	\nabla f (y_{1} )-\nabla f (y_{2}) \in \aff  (\mathcal{S})   -\aff    (\mathcal{S}) =M.
	\end{align}	
	Because $\{y_{1}, y_{2} \}  \subseteq  \overrightarrow{E}_{f}(\mathcal{S}) $,    \cref{lemma:nablafyi:in:Mperp} implies that
	\begin{align}\label{eq:lem:CCOfrS:Mperp}
	\nabla f (y_{1} )-\nabla f (y_{2})  \in  M^{\perp}.
	\end{align}
	Combine	\cref{eq:lem:CCOfrS:M} and \cref{eq:lem:CCOfrS:Mperp} to deduce that $\nabla f (y_{1} )-\nabla f (y_{2})  \in  M \cap M^{\perp} =\{0\}$, that is,
	\begin{align}\label{eq:lem:CCOfrS:EQ}
	\nabla f (y_{1} )= \nabla f (y_{2}).
	\end{align}
	Apply $\nabla f^{*} $ in both sides of \cref{eq:lem:CCOfrS:EQ} and utilize  \cref{fact:nablaf:nablaf*:id}  to obtain that $y_{1} =y_{2}$.
\end{proof}

We are now ready to define forward Bregman circumcenters and pseudo-circumcenters.
\begin{definition} \label{defn:CCS:Bregman:forward}
	Let $\mathcal{P} (\dom f)$ be the set of all nonempty subsets of $\dom f$ containing finitely many elements.  For every $K \in \mathcal{P} (\dom f)$,
	set
	\begin{empheq}[box=\mybluebox]{equation*}
	\overrightarrow{E}_{f}(K):= \{ q \in \inte \dom f~:~  (\forall x \in K) ~\D_{f} (x,q) \text{ is a singleton} \}.
	\end{empheq}
	\begin{enumerate}
		\item \label{defn:CCS:Bregman:forward:}	Define the \emph{forward Bregman circumcenter  operator w.r.t.\,$f$} as
			\begin{empheq}[box=\mybluebox]{equation*}
	\overrightarrow{\CCO{}}_{f}  :  	\mathcal{P} (\dom f) \to 2^{\mathcal{H}} : K \mapsto  \aff (  K ) \cap 	\overrightarrow{E}_{f}(K).
		\end{empheq}
		\item \label{defn:CCS:Bregman:forward:ps} Define the \emph{forward Bregman pseudo-circumcenter  operator w.r.t.\,$f$} as
		\begin{empheq}[box=\mybluebox]{equation*}
		\overrightarrow{\CCO{}}^{ps}_{f} :  \mathcal{P} (\dom f)  \to  2^{\mathcal{H}}  : K \mapsto
	\left(	\nabla f^{*} \left( \aff   K \right) \right) \cap \overrightarrow{E}_{f}  (K).
		\end{empheq}
	\end{enumerate}
In particular, for every $K \in  \mathcal{P} (\dom f)$, we call the element  in $\overrightarrow{\CCO{}}_{f}(K) $ and   $\overrightarrow{\CCO{}}^{ps}_{f}(K)$ forward Bregman circumcenter  and forward Bregman pseudo-circumcenter of $K$, respectively.
\end{definition}

In view of \cref{lem:CCOfrS}, $\left( \forall K \in \mathcal{P} (\dom f) \right)$ $\overrightarrow{\CCO{}}^{ps}_{f}(K)$
 is either a singleton or an empty set.

\subsection*{Existence of forward Bregman circumcenters}

\begin{theorem} \label{theorem:forwardCCS}
	 	Suppose that $f$ is Legendre and that $\overrightarrow{E}_{f}(\mathcal{S} )   \neq \varnothing$.  Let $y \in \overrightarrow{E}_{f}(\mathcal{S} )$.
	Then the following hold.
	\begin{enumerate}
		\item \label{theorem:forwardCCS:EQ} $\overrightarrow{\CCO{}}_{f}(\mathcal{S}) = \aff (\mathcal{S}) \cap \inte \dom f \cap \nabla f^{*} \left( \nabla f(y) +M^{\perp}  \right)$.
		\item \label{theorem:forwardCCS:P} Suppose that $ \aff (\mathcal{S}) \cap  \inte \dom f \neq \varnothing$ $($e.g., $\dom f =\mathcal{H}$ or $ \mathcal{S} \cap  \inte \dom f \neq \varnothing$$)$.  Then $ \overleftarrow{\Pro}^{f}_{\aff(\mathcal{S})}(y) \in \overrightarrow{\CCO{}}_{f}(\mathcal{S})$.
	\end{enumerate}
\end{theorem}

\begin{proof}
	\cref{theorem:forwardCCS:EQ}: This is clear from \cref{defn:CCS:Bregman:forward}\cref{defn:CCS:Bregman:forward:}, \cref{lemma:nablafyi:in:Mperp}, and \cref{fact:nablaf:nablaf*:id}.
	
	\cref{theorem:forwardCCS:P}:
Because $y \in \overrightarrow{E}_{f}(\mathcal{S}) \subseteq  \inte \dom f$,  $\aff (\mathcal{S}) \cap \inte \dom f \neq \varnothing$, and  $\aff (\mathcal{S})$ is a closed affine subspace, apply \cref{theor:affine:character:P}\cref{theor:affine:character:PleftCf} with $U=\aff (\mathcal{S})$ to obtain that
	\begin{subequations}
			\begin{align}
		&\overleftarrow{\Pro}^{f}_{\aff (\mathcal{S})} (y)   \in \aff (\mathcal{S}) \cap \inte \dom f, \text{ and} \label{eq:theorem:leftP:in:rightCCS:P}\\
	&	\Innp{\nabla f(y) -\nabla f \left( \overleftarrow{\Pro}^{f}_{\aff (\mathcal{S})} (y)  \right), \aff (\mathcal{S})-\overleftarrow{\Pro}^{f}_{\aff (\mathcal{S})} (y)  } = 0.\label{eq:theorem:leftP:in:rightCCS:innp}
		\end{align}
	\end{subequations}
Employ $\overleftarrow{\Pro}^{f}_{\aff (\mathcal{S})} (y)   \in \aff (\mathcal{S}) $ and \cref{eq:M} to see that $  \aff (\mathcal{S})-\overleftarrow{\Pro}^{f}_{\aff (\mathcal{S})} (y)  =M$. Combine this with \cref{eq:theorem:leftP:in:rightCCS:innp} and \cref{fact:nablaf:nablaf*:id} to observe  that
\begin{align*}
\nabla f(y) -\nabla f \left( \overleftarrow{\Pro}^{f}_{\aff (\mathcal{S})} (y)  \right) \in M^{\perp} \Leftrightarrow \overleftarrow{\Pro}^{f}_{\aff (\mathcal{S})} (y)  \in ( \nabla f)^{-1} \left( \nabla f(y) +M^{\perp}  \right) =\nabla f^{*} \left( \nabla f(y) +M^{\perp}  \right),
\end{align*}
which, combining with \cref{eq:theorem:leftP:in:rightCCS:P} and \cref{theorem:forwardCCS:EQ}, yields  the desired result.
\end{proof}

	\begin{example}\label{exam:forward:Rn}
		Suppose that $\mathcal{H} =\mathbb{R}^{n}$.  Denote by  $(\forall x \in \mathcal{H} )$ $x= (x_{i})^{n}_{i=1} $.  Then the following statements hold.
		\begin{enumerate}
			\item \label{exam:forward:Rn:negative} Suppose that  $(\forall x \in \left[0, +\infty\right[^{n} )$ $f(x) =\sum^{n}_{i=1} x_{i} \ln( x_{i}) -x_{i}$.
			\begin{enumerate}
				\item \label{exam:forward:Rn:negative:xyz} Suppose that $\mathcal{S}:= \{x,y,z\} \subseteq \left]0, +\infty\right[^{n}$ such that $x$,  $y$ and $ z$ are pairwise distinct. Let $p \in \mathcal{H} $.  Then $ p \in 	\overrightarrow{\CCO{}}_{f}(\mathcal{S}) $ if and only if $p=x + \alpha (y-x) +\beta (z-x) \in \left]0, +\infty\right[^{n}$, where $(\alpha, \beta)^{\intercal} \in \mathbb{R}^{2}$ solves the following system of equations
				\begin{align*}
			\begin{cases}
			\sum^{n}_{i=1} (y_{i}-x_{i}) \ln\left(  x_{i} +\alpha (y_{i} -x_{i}) +\beta (z_{i} -x_{i})  \right)=f(y) -f(x)\\
			\sum^{n}_{i=1} (z_{i}-x_{i}) \ln\left(  x_{i} +\alpha (y_{i} -x_{i}) +\beta (z_{i} -x_{i})  \right)=f(z) -f(x)
			\end{cases}.
				\end{align*}
				\item \label{exam:forward:Rn:negative:} Suppose that $\mathcal{S}:= \{(1,1,1), (1,2,1), (1,1,2)\} $. Then
				\begin{align*}
				\overrightarrow{\CCO{}}_{f}(\mathcal{S}) = \left\{ (1,1,1) + \left( \frac{4}{\rm e} -1\right) (0,1,0) +\left( \frac{4}{\rm e} -1\right) (0,0,1)   \right\}.
				\end{align*}
			\end{enumerate}
			\item \label{exam:forward:Rn:FD}	Suppose that $(\forall x \in \left[0, 1\right]^{n} )$ $f(x) =\sum^{n}_{i=1} x_{i} \ln( x_{i}) + (1-x_{i}) \ln (1-x_{i})$.
			\begin{enumerate}
				\item  \label{exam:forward:Rn:FD:xyz} Suppose that $\mathcal{S}:= \{x,y,z\} \subseteq  \left]0, 1\right[^{n} $ such that $x$,  $y$ and $ z$ are pairwise distinct. Let $p \in \mathcal{H} $.  Then $ p \in \overrightarrow{\CCO{}}_{f}(\mathcal{S})  $ if and only if $p=x + \alpha (y-x) +\beta (z-x) \in \left]0, 1\right[^{n}$, where $(\alpha, \beta)^{\intercal} \in \mathbb{R}^{2}$ solves the following system of equations
				\begin{align*}
				 	\begin{cases}
				 \sum^{n}_{i=1} (y_{i}-x_{i}) \ln\left(  \frac{ x_{i} +\alpha (y_{i} -x_{i}) +\beta (z_{i} -x_{i})  }{1-x_{i} -\alpha (y_{i} -x_{i}) -\beta (z_{i} -x_{i}) } \right)=f(y) -f(x)\\
				 \sum^{n}_{i=1} (z_{i}-x_{i}) \ln\left(  \frac{ x_{i} +\alpha (y_{i} -x_{i}) +\beta (z_{i} -x_{i})  }{1-x_{i} -\alpha (y_{i} -x_{i}) -\beta (z_{i} -x_{i}) } \right)=f(z) -f(x)
				 \end{cases}.
				\end{align*}
				\item  \label{exam:forward:Rn:FD:} Suppose that $\mathcal{S}:= \{(\frac{1}{4},\frac{1}{4},\frac{1}{4}), (\frac{1}{4},\frac{1}{2},\frac{1}{4}), (\frac{1}{4},\frac{1}{4},\frac{1}{2})\} $. Then
				\begin{align*}
			\overrightarrow{\CCO{}}_{f}(\mathcal{S})=  \left\{   \left(\frac{1}{4},\frac{1}{4},\frac{1}{4}\right) +\frac{21}{43} \left(0,\frac{1}{4},0\right)+\frac{21}{43} \left(0,0,\frac{1}{4}\right)
			 \right\}= \left\{ \left(\frac{1}{4},\frac{16}{43},\frac{16}{43}\right)    \right\}.
				\end{align*}
			\end{enumerate}	
		\end{enumerate}
	\end{example}
	
	\begin{proof}
		\cref{exam:forward:Rn:negative:xyz}:  As a consequence of \cref{defn:BregmanDistance},
		\begin{align*}
		D_{f} (p,x) =D_{f} (p,y)=D_{f} (p,z)  &\Leftrightarrow  \begin{cases}
	\innp{\nabla f(p), y-x} =f(y) -f(x)\\
	 \innp{\nabla f(p),z-x} =f(z) -f(x)
		\end{cases}\\
		& \Leftrightarrow \begin{cases}
		\sum^{n}_{i=1} (y_{i}-x_{i}) \ln(p_{i}) =f(y) -f(x)\\
		\sum^{n}_{i=1} (z_{i}-x_{i}) \ln(p_{i}) =f(z) -f(x)
		\end{cases}\\
		& \Leftrightarrow
		\begin{cases}
		 \sum^{n}_{i=1} (y_{i}-x_{i}) \ln\left(  x_{i} +\alpha (y_{i} -x_{i}) +\beta (z_{i} -x_{i})  \right)=f(y) -f(x)\\
		\sum^{n}_{i=1} (z_{i}-x_{i}) \ln\left(  x_{i} +\alpha (y_{i} -x_{i}) +\beta (z_{i} -x_{i})  \right)=f(z) -f(x)
		\end{cases}.
		\end{align*}
		
		 \cref{exam:forward:Rn:negative:}: This follows directly from \cref{exam:forward:Rn:negative:xyz} above.
		
		\cref{exam:forward:Rn:FD}: The proof is similar to that of \cref{exam:forward:Rn:negative} and is omitted here.
	\end{proof}

\begin{example} \label{exam:forward:R}
	Suppose that $\mathcal{H}=\mathbb{R}$. Then the following statements hold.
	\begin{enumerate}
		\item  \label{exam:forward:R:negative} Suppose that  $(\forall x \in \left[0, +\infty\right[ )$ $f(x) =x \ln( x) -x$.
		\begin{enumerate}
			\item  \label{exam:forward:R:negative:xy} Suppose $\mathcal{S}:= \{x,y\} \subseteq \left[0, +\infty\right[$ with $x \neq y$. Then $\overrightarrow{\CCO{}}_{f}(\mathcal{S}) = \left\{ \exp \left(  \frac{y \ln(y) -x\ln(x) +x -y}{y-x}  \right) \right\}$.
			\item  \label{exam:forward:R:negative:xyz} Suppose $\mathcal{S}:= \{x,y,z\} \subseteq \left[0, +\infty\right[$ such that $x$,  $y$, and $ z$ are pairwise distinct. Then $\overrightarrow{\CCO{}}_{f}(\mathcal{S}) = \varnothing$.
		\end{enumerate}
		
		\item  \label{exam:forward:R:FD} Suppose that $(\forall x \in \left[0, 1\right] )$ $f(x) =x \ln( x) + (1-x) \ln (1-x)$.
		\begin{enumerate}
			\item \label{exam:forward:R:FD:xy} Suppose $\mathcal{S}:= \{x,y\} \subseteq \left[0, 1\right] $ with $x \neq y$. Then
			\begin{align*}
			\overrightarrow{\CCO{}}_{f}(\mathcal{S}) = \left\{  \left(  \exp \left( \frac{  x\ln(x) + (1-x) \ln(1-x) - y\ln(y) -(1-y) \ln(1-y) }{y-x} \right) +1  \right)^{-1} \right\} .
			\end{align*}
			
			\item \label{exam:forward:R:FD:xyz} Suppose $\mathcal{S}:= \{x,y,z\} \subseteq  \left[0, 1\right] $ such that $x$,  $y$, and $ z$ are pairwise distinct. Then $\overrightarrow{\CCO{}}_{f}(\mathcal{S}) = \varnothing$.
		\end{enumerate}
	\end{enumerate}
	
\end{example}

\begin{proof}
	This comes from  \cref{defn:BregmanDistance} and some easy calculus and algebra.  In particular, the proof of \cref{exam:forward:R:negative:xyz}  and \cref{exam:forward:R:FD:xyz}  are similar to that of \cref{exam:back:R}.
\end{proof}

\subsection*{Explicit formula of forward Bregman pseudo-circumcenters}

\begin{proposition}  \label{prop:psuCCS:forward:matrix}
	Set
	\begin{align*}
	&A:=
	\begin{pmatrix}
	\Innp{ p_{1} -p_{0}, p_{1} -p_{0} } & \ldots& \Innp{p_{1} -p_{0},p_{m} -p_{0}}\\
	\vdots & \ldots & \vdots\\
	\Innp{p_{m} -p_{0}, p_{1} -p_{0}} & \ldots& \Innp{p_{m} -p_{0}, p_{m} -p_{0} }\\
	\end{pmatrix} \text{ and }
 B:= \begin{pmatrix}
	\beta_{1} - \Innp{p_{0}, p_{1}-p_{0}}\\
	\vdots\\
	\beta_{m} - \Innp{p_{0}, p_{m}-p_{0}}
	\end{pmatrix}.
	\end{align*}
	Then
	$ \overrightarrow{\CCO{}}^{ps}_{f}(\mathcal{S}) = \left\{ q \in \inte \dom f   ~:~ \exists (\alpha_{1}, \ldots, \alpha_{m} )^{\intercal}  \in \mathbb{R}^{m} \text{ s.t. } \nabla f (q) =p_{0} +\sum^{m}_{i=1} \alpha_{i} \left( p_{i} -p_{0} \right)  \text{ and } A\alpha =B   \right\}$.
	
\end{proposition}

\begin{proof}
	The proof is similar to that of \cref{prop:formualCCS:Pleft:matrixEQ}, while this time we utilize
	\cref{defn:CCS:Bregman:forward}\cref{defn:CCS:Bregman:forward:ps}  and \cref{lemma:xinEp1pm:right}.
\end{proof}

\begin{theorem} \label{theorem:psuCCS:forward}
Suppose that $f$ is Legendre and $ \overrightarrow{E}_{f}(\mathcal{S})  \neq \varnothing$.   Let $y \in\overrightarrow{E}_{f}(\mathcal{S})   $.
Then the following  hold.
\begin{enumerate}
	\item \label{theorem:psuCCS:forward:cap} $  \overrightarrow{\CCO{}}^{ps}_{f}(\mathcal{S})  = \nabla f^{*} \left( \aff (\mathcal{S})\right) \cap \inte \dom f \cap \nabla f ^{*}(\nabla f (y) + M^{\perp})$.

	\item \label{theorem:psuCCS:forward:P} Suppose that $ \Pro_{\aff( \mathcal{S}  )} \left( \nabla f (y) \right) \in \inte \dom f^{*}$ $($e.g., $\aff( \mathcal{S}  ) \subseteq  \inte \dom f^{*}$ or $\dom f^{*} =\mathcal{H}$$)$. Then $\overrightarrow{\CCO{}}^{ps}_{f}(\mathcal{S})= \nabla f^{*} \left( \Pro_{\aff( \mathcal{S}  )} (\nabla f(y)) \right)$.
\end{enumerate}

\end{theorem}

\begin{proof}
	\cref{theorem:psuCCS:forward:cap}: This is clear from \cref{defn:CCS:Bregman:forward}\cref{defn:CCS:Bregman:forward:}, 	\cref{fact:nablaf:nablaf*:id} and \cref{lemma:nablafyi:in:Mperp}.

\cref{theorem:psuCCS:forward:P}: Exploit	\cref{fact:nablaf:nablaf*:id} to observe that $ \nabla f^{*} \left( \Pro_{\aff( \mathcal{S}  )} (\nabla f(y)) \right) \in  \inte \dom f \cap \nabla f^{*} (\aff(\mathcal{S}))$.
Moreover,
	\begin{align*}
	\Pro_{\aff( \mathcal{S}  )} (\nabla f(y)) -\nabla f(y)   &= \Pro_{p_{0} +M} (\nabla f(y)) - \nabla f(y) \quad (\text{by \cref{eq:M}})\\
	&= p_{0}  + \Pro_{ M} (\nabla f(y)-p_{0} ) -\nabla f(y) \quad (\text{by \cite[Proposition~3.19]{BC2017}})\\
	&=- \Pro_{ M^{\perp}} (\nabla f(y)-p_{0}  ) \in M^{\perp},   \quad (\text{by \cite[Theorem~5.8]{D2012}})
	\end{align*}
	which, due to \cref{eq:M}, implies that $\nabla f^{*} \Pro_{\aff( \mathcal{S}  )} (\nabla f(y)) \in  \nabla f ^{*}(\nabla f (y) + M^{\perp})$.
Altogether, the desired result follows from \cref{theorem:psuCCS:forward:cap} above.
\end{proof}

\begin{corollary} \label{cor:equi:Ef:CCfS:right}
Suppose that $f$ is Legendre and that $  \Pro_{\aff( \mathcal{S}  )} \circ \nabla f  \left(  \overrightarrow{E}_{f}(\mathcal{S})  \right) \subseteq \inte \dom f^{*}$ $($e.g., $\dom f^{*} =\mathcal{H}$$)$.  Then $\overrightarrow{E}_{f}(\mathcal{S}) \neq \varnothing$ if and only if
	  $\overrightarrow{\CCO{}}^{ps}_{f}(\mathcal{S}) \neq \varnothing$.
\end{corollary}

\begin{proof}
 This follows easily  from  \cref{defn:CCS:Bregman:forward}\cref{defn:CCS:Bregman:forward:ps}.
and \cref{theorem:psuCCS:forward}\cref{theorem:psuCCS:forward:P}.
\end{proof}

The following result provides an explicit formula for forward Bregman pseudo-circumcenters.
\begin{theorem} \label{thm:LinIndpPformula:TpseudoCCS}
	Suppose that $f$ is Legendre, that $p_{0}, p_{1}, \ldots, p_{m}$ are affinely independent and that $\aff ( \mathcal{S}) \subseteq \inte \dom f^{*}$.
	Then $\overrightarrow{\CCO{}}^{ps}_{f}(\mathcal{S}) \neq \varnothing$.
	Moreover,
	\begin{align} \label{EQ:thm:LinIndpPformula:TpseudoCCS:formula}
	\overrightarrow{\CCO{}}^{ps}_{f}(\mathcal{S})= \nabla f^{*} \left( p_{0}+\alpha_{1}(p_{1}-p_{0} )+ \cdots+ \alpha_{m}(p_{m}-p_{0}) \right),
	\end{align}
	where
	\begin{align} \label{eq:thm:LinIndpPformula:TpseudoCCS}
	\begin{pmatrix}
	\alpha_{1} \\
	\vdots\\
	\alpha_{m} \\
	\end{pmatrix}
	:=
	G( p_{1}-p_{0} ,\ldots,p_{m}-p_{0} )^{-1}
	\begin{pmatrix}
	\eta_{1} -\innp{p_{0}, p_{1}-p_{0} }\\
	\vdots\\
	\eta_{m}-\innp{p_{0}, p_{m}-p_{0}} \\
	\end{pmatrix}.
	\end{align}
\end{theorem}

\begin{proof}
	Denote by $\bar{z} :=p_{0}+\alpha_{1}(p_{1}-p_{0} )+ \cdots+ \alpha_{m}(p_{m}-p_{0})$,
	where $(\alpha_{1}, \ldots, \alpha_{m})^{\intercal}$ is defined as   \cref{eq:thm:LinIndpPformula:TpseudoCCS} above.
	According to \cref{lemma:CCSpsF},
	\begin{align} \label{eq:thm:LinIndpPformula:TpseudoCCS:S}
	\left\{ x \in \aff ( \mathcal{S})  ~:~  (\forall i \in \I) ~ \innp{x , p_{i} -p_{0}} = \eta_{i}  \right\}  = \{\bar{z}\}.
	\end{align}
	
	Notice that,  by \cref{fact:nablaf:nablaf*:id}, $\bar{z} \in \aff ( \mathcal{S}) \subseteq \inte \dom f^{*}$ implies that $\nabla f^{*} (\bar{z}) \in  \inte \dom f$ and $\bar{z} =\nabla f \left( \nabla f^{*} (\bar{z}) \right)$.
Because $\nabla f^{*} \left( \aff ( \mathcal{S}) \right)  \cap \inte \dom f = \nabla f^{*} \left( \aff ( \mathcal{S}) \right)$, applying \cref{fact:nablaf:nablaf*:id} and \cref{eq:thm:LinIndpPformula:TpseudoCCS:S}, we obtain that
\begin{align}\label{eq:thm:LinIndpPformula:TpseudoCCS:y}
\left\{ y \in \nabla f^{*} \left( \aff ( \mathcal{S}) \right)  \cap \inte \dom f~:~  (\forall i \in \I) ~ \innp{\nabla f (y) , p_{i} -p_{0}} = \eta_{i}  \right\}  = \{ \nabla f^{*} (\bar{z}) \}.
\end{align}
	On the other hand, due to \cref{defn:CCS:Bregman:forward}\cref{defn:CCS:Bregman:forward:ps} and  \cref{lemma:xinEp1pm:right},
	\begin{align}\label{eq:thm:LinIndpPformula:TpseudoCCS:CCS}
	\overrightarrow{\CCO{}}^{ps}_{f}(\mathcal{S})=	\left\{ y \in \nabla f^{*} \left( \aff ( \mathcal{S}) \right)  \cap \inte \dom f~:~  (\forall i \in \I) ~ \innp{\nabla f (y) , p_{i} -p_{0}} = \eta_{i}  \right\}.
	\end{align}
	Clearly, \cref{eq:thm:LinIndpPformula:TpseudoCCS:y} and \cref{eq:thm:LinIndpPformula:TpseudoCCS:CCS} entail the required result.
\end{proof}

\section{Duality correspondence} \label{sec:Miscellaneous}
Duality is the key for connections between  backward and forward Bregman (pseudo-)circumcenters.

Let $f \in \Gamma_{0} (\mathcal{H}) $ be G\^ateaux  differentiable  on  $\inte \dom f \neq \varnothing$. Suppose that
 \begin{empheq}[box=\mybluebox]{equation*}
 \mathcal{S}:= \{q_{0}, q_{1}, \ldots, q_{m} \} \subseteq    \dom f \text{ and }  \mathcal{S} \text{ is nonempty}.
 \end{empheq}
 Set
\begin{align} \label{eq:R:Eright}
\overrightarrow{E}_{f}(\mathcal{S} ):= \{ y \in  \inte \dom f~:~  \D_{f} (q_{0},y) =\D_{f} (q_{1},y) =\cdots =\D_{f} (q_{m},y) \}.
\end{align}
If $\mathcal{S} \subseteq \inte \dom f$, we set
 \begin{align} \label{eq:R:Ef}
\overleftarrow{E}_{f}(\mathcal{S}):= \{ x \in \dom f ~:~  \D_{f} (x,q_{0}) =\D_{f} (x, q_{1}) =\cdots =\D_{f} (x, q_{m}) \}.
\end{align}


\begin{theorem} \label{theor:CCS:Rel}
	Suppose that $f$ is Legendre and that $\mathcal{S} \subseteq \inte \dom f$. Then the following hold.
	\begin{enumerate}
		\item \label{theor:CCS:Rel:E} $\overleftarrow{E}_{f}(\mathcal{S}) \cap \inte \dom f= \nabla f^{*} \left( \overrightarrow{E}_{f^{*}}\left( \nabla f \left( \mathcal{S} \right)\right) \right)$.
		\item \label{theor:CCS:Rel:forwardE} $\overrightarrow{E}_{f}(\mathcal{S}) = \nabla f^{*} \left(   \overleftarrow{E}_{f^{*}}\left( \nabla f \left( \mathcal{S}\right) \right) \cap \inte \dom f^{*}  \right)$.
		\item \label{theor:CCS:Rel:CCS}  $	\overleftarrow{\CCO{}}^{ps}_{f} ( \mathcal{S} )  \cap \inte \dom f = \nabla f^{*} \left(  	\overrightarrow{\CCO{}}^{ps}_{f^{*}} \left( \nabla f (\mathcal{S} )\right) \right)$.
		\item \label{theor:CCS:Rel:forwardCCS} $ \overrightarrow{\CCO{}}_{f}(\mathcal{S}) = \aff (\mathcal{S}) \cap \nabla f^{*} \left(   \overleftarrow{E}_{f^{*}}\left( \nabla f \left( \mathcal{S}\right) \right)   \cap \inte \dom f^{*}  \right) = \nabla f^{*} \left(   \nabla f \left(    \aff (\mathcal{S})  \right)  \cap  \overleftarrow{E}_{f^{*}}\left( \nabla f \left( \mathcal{S}\right) \right)    \right) $.
	\end{enumerate}
\end{theorem}

\begin{proof}
\cref{theor:CCS:Rel:E}: Because $f$ is Legendre, due to \cref{def:Legendre} and   \cref{fact:PropertD},
	\begin{align} \label{prop:Rel:E:D}
(\forall x \in \inte \dom f) ~(\forall i \in \I \cup \{0\}) \quad \D_{f} (x,q_{i})   = D_{f^{*}}(\nabla f (q_{i}),\nabla f (x)).
	\end{align}
	On the other hand, by \cref{eq:R:Eright},
	\begin{align}\label{prop:Rel:E:f*}
	\overrightarrow{E}_{f^{*}} \left( \nabla f \left(\mathcal{S} \right) \right)= \left\{ y \in  \inte  \dom f^{*}~:~  (\forall i \in \I) ~ \D_{f^{*}} \left(\nabla f (q_{0}),y \right)  =\D_{f^{*}} \left(\nabla f (q_{i}),y \right) \right\}.
	\end{align}
	Let $x \in \mathcal{H}$. Then
	\begin{align*}
&	x \in \overleftarrow{E}_{f}(\mathcal{S})  \cap \inte \dom f \\
\Leftrightarrow & x \in \inte \dom f \text{ and } (\forall i \in \I) ~ \D_{f} (x,q_{0}) =\D_{f} (x, q_{i}) \quad (\text{by \cref{eq:R:Ef}})\\
 	\Leftrightarrow & x \in \inte \dom f \text{ and }  (\forall i \in \I) ~D_{f^{*}}(\nabla (q_{0}),\nabla f (x)) =D_{f^{*}}(\nabla (q_{i}),\nabla f (x)) \quad (\text{by \cref{prop:Rel:E:D}})\\
 	\Leftrightarrow	& \nabla f (x) \in \inte  \dom f^{*} \text{ and }  (\forall i \in \I) ~ D_{f^{*}}(\nabla (q_{0}),\nabla f (x)) =D_{f^{*}}(\nabla (q_{i}),\nabla f (x)) \quad (\text{by \cref{fact:nablaf:nablaf*:id}})\\
 	\Leftrightarrow & \nabla f (x) \in \overrightarrow{E}_{f^{*}} \left( \nabla f \left(\mathcal{S} \right) \right) \quad (\text{by \cref{prop:Rel:E:f*}})\\
 	\Leftrightarrow	 & x \in  \nabla f^{*} \left( \overrightarrow{E}_{f^{*}}\left( \nabla f \left( \mathcal{S} \right)\right) \right), \quad (\text{by \cref{fact:nablaf:nablaf*:id}})
	\end{align*}
which verifies \cref{theor:CCS:Rel:E}.

\cref{theor:CCS:Rel:forwardE}:  The proof is similar to that of  \cref{theor:CCS:Rel:E}.

\cref{theor:CCS:Rel:CCS}: Based on	\cref{defn:CCS:Bregman:left}\cref{defn:CCS:Bregman:left:ps} and \cref{defn:CCS:Bregman:forward}\cref{defn:CCS:Bregman:forward:ps},
\begin{align*}
\overleftarrow{\CCO{}}^{ps}_{f} ( \mathcal{S} )  \cap \inte \dom f&=   \aff ( \nabla f (\mathcal{S} ))\cap \overleftarrow{E}_{f}( \mathcal{S} ) \cap \inte \dom f\\
&=   \aff ( \nabla f (  \mathcal{S} ))\cap \nabla f^{*} \left( \overrightarrow{E}_{f^{*}}\left( \nabla f \left( \mathcal{S} \right)\right) \right) \quad (\text{by  \cref{theor:CCS:Rel:E} above})\\
&=  \nabla f^{*} \left( \nabla f \left(   \aff ( \nabla f (\mathcal{S} )) \right) \right) \cap \nabla f^{*} \left( \overrightarrow{E}_{f}\left( \nabla f \left( \mathcal{S} \right)\right) \right)  \quad \left(\text{\cref{fact:nablaf:nablaf*:id} implies $\nabla f^{*}\nabla f =\Id$}\right)\\
&=  \nabla f^{*} \left( \nabla f \left(   \aff ( \nabla f (\mathcal{S} )) \right) \cap \overrightarrow{E}_{f^{*}}\left( \nabla f \left( \mathcal{S} \right)\right)  \right) \quad \left(\text{\cref{fact:nablaf:nablaf*:id} states $\nabla f^{*}$ is bijective}\right)\\
&=	 \nabla f^{*}  \left( 	\overrightarrow{\CCO{}}^{ps}_{f^{*}}  ( \nabla f (\mathcal{S} )) \right).
\end{align*}

\cref{theor:CCS:Rel:forwardCCS}: This follows immediately from \cref{defn:CCS:Bregman:forward}\cref{defn:CCS:Bregman:forward:}, \cref{fact:nablaf:nablaf*:id}, and    \cref{theor:CCS:Rel:forwardE} above.
\end{proof}

\begin{corollary} \label{cor:CCS:Rel}
	Suppose that $f$ is Legendre with $\dom f$ and $\dom f^{*} $ being open. Then the following hold.
	\begin{enumerate}
		\item \label{cor:CCS:Rel:E} $\overleftarrow{E}_{f}(\mathcal{S}) = \nabla f^{*} \left( \overrightarrow{E}_{f^{*}}\left( \nabla f \left( \mathcal{S} \right)\right) \right)$.
		\item \label{cor:CCS:Rel:Eforward} $\overrightarrow{E}_{f}(\mathcal{S}) = \nabla f^{*} \left(   \overleftarrow{E}_{f^{*}}\left( \nabla f \left( \mathcal{S}\right) \right)  \right)$.
		\item \label{cor:CCS:Rel:CCS}  $	\overleftarrow{\CCO{}}^{ps}_{f} ( \mathcal{S} ) = \nabla f^{*} \left(  	\overrightarrow{\CCO{}}^{ps}_{f^{*}} \left( \nabla f (\mathcal{S} )\right) \right)$.
	\end{enumerate}
\end{corollary}

\begin{proof}
	Taking \cref{eq:R:Ef} and \cref{defn:CCS:Bregman:left}\cref{defn:CCS:Bregman:left:ps} into account, we know that under the hypothesis that $\dom f$ and $\dom f^{*} $ are open, $\overleftarrow{E}_{f}(\mathcal{S}) \cap \inte \dom f = \overleftarrow{E}_{f}(\mathcal{S})$, $\overleftarrow{E}_{f^{*}}\left( \nabla f \left( \mathcal{S}\right) \right) \cap \inte \dom f^{*} = \overleftarrow{E}_{f^{*}}\left( \nabla f \left( \mathcal{S}\right) \right)$, and $	\overleftarrow{\CCO{}}^{ps}_{f} ( \mathcal{S} )  \cap \inte \dom f = 	\overleftarrow{\CCO{}}^{ps}_{f} ( \mathcal{S} ) $. Therefore, the result follows immediately from \cref{theor:CCS:Rel}.
\end{proof}


Although there are more generalizations of the classical circumcenter under Bregman distances,
the following results suggest that they can probably be deduced by Bregman (pseudo-) circumcenters defined in   \cref{defn:CCS:Bregman:left,defn:CCS:Bregman:forward}.
\begin{proposition} \label{prop:newCCS}
 Suppose that $f$ is Legendre.  Define
	\begin{align*}
	\stackrel{\mathlarger{\twoheadleftarrow}}{\CCO{}}_{f}  : \mathcal{P}(  \inte \dom f) \to  2^{\mathcal{H}} :  \mathcal{S}
 \mapsto
	\nabla f^{*} \left(    \aff \left(  \nabla f (\mathcal{S}) \right)  \right) \cap \overleftarrow{E}_{f}\left(\mathcal{S} \right).
	\end{align*}
Then $\left( \forall \mathcal{S} \in \mathcal{P}(  \inte  \dom f)  \right)$ $\stackrel{\mathlarger{\twoheadleftarrow}}{\CCO{}}_{f}(\mathcal{S})
=  \nabla f^{*} \left(   \overrightarrow{\CCO{}}_{f^{*}}\left( \nabla f (\mathcal{S}) \right) \right)$.
\end{proposition}

\begin{proof}
Employing \cref{defn:CCS:Bregman:forward}\cref{defn:CCS:Bregman:forward:}, \cref{fact:nablaf:nablaf*:id}, and \cref{theor:CCS:Rel}\cref{theor:CCS:Rel:E}, we observe that for every $\mathcal{S} \in  \mathcal{P}( \inte \dom f) $,
\begin{align*}
 \nabla f^{*} \left(   \overrightarrow{\CCO{}}_{f^{*}}\left( \nabla f (\mathcal{S}) \right) \right)
 &~=~ \nabla f^{*} \left(     \aff   \left( \nabla f (\mathcal{S}) \right)
 \cap 	\overrightarrow{E}_{f^{*}}\left( \nabla f (\mathcal{S}) \right) \right) \\
& ~=~ \nabla f^{*} \left(     \aff   \left( \nabla f (\mathcal{S}) \right) \right)
\cap \nabla f^{*} \left( 	\overrightarrow{E}_{f^{*}}\left( \nabla f (\mathcal{S}) \right) \right)\\
 & ~=~ \nabla f^{*} \left(     \aff   \left( \nabla f (\mathcal{S}) \right) \right) \cap \overleftarrow{E}_{f}(\mathcal{S})
 \cap \inte \dom f\\
&= \nabla f^{*} \left(     \aff   \left( \nabla f (\mathcal{S}) \right) \right) \cap \overleftarrow{E}_{f}(\mathcal{S}) \\
 & ~=~ \, \stackrel{\mathlarger{\twoheadleftarrow}}{\CCO{}}_{f}(\mathcal{S}).
\end{align*}
\end{proof}

\begin{proposition} \label{prop:newCCSps}
	 Suppose that $f$ is  Legendre. Define
	\begin{align*}
		\stackrel{\mathlarger{\twoheadleftarrow}}{\CCO{}}^{ps}_{f}  : \mathcal{P}( \inte \dom f) \to  2^{\mathcal{H}} :
\mathcal{S} \mapsto
	\aff (\mathcal{S})\cap \overleftarrow{E}_{f^*}( \nabla f (\mathcal{S}) ).
	\end{align*}
	Then $\stackrel{\mathlarger{\twoheadleftarrow}}{\CCO{}}^{ps}_{f}(\mathcal{S}) = \overleftarrow{\CCO{}}^{ps}_{f^*}(\nabla f(\mathcal{S}))$.
\end{proposition}

\begin{proof}
	According to \cref{fact:nablaf:nablaf*:id}, $\left(\forall \mathcal{S} \in \mathcal{P}( \inte \dom f) \right)$ $
\mathcal{S}= \nabla f^* \left( \nabla f(\mathcal{S}) \right)$. Hence, by \cref{defn:CCS:Bregman:left}\cref{defn:CCS:Bregman:left:ps}, we obtain that
	\begin{align*}
	\left(\forall \mathcal{S} \in \mathcal{P}( \inte \dom f) \right) \quad  \stackrel{\mathlarger{\twoheadleftarrow}}{\CCO{}}^{ps}_{f} (\mathcal{S})
 = 	\aff \left( \nabla f^* ( \nabla f(\mathcal{S})) \right) \cap \overleftarrow{E}_{f^*}( \nabla f(\mathcal{S}))
 =  \overleftarrow{\CCO{}}^{ps}_{f^*} (\nabla f(\mathcal{S})).
	\end{align*}
\end{proof}

\section{A comparison of classical circumcenters and Bregman circumcenters}\label{s:compare}

Under simplified assumptions, we summarize results in Section~\ref{sec:ForwardBregmancircumcenters} and Section~\ref{sec:BackwardBregmancircumcenters}
on the existence and uniqueness of backward and forward Bregman (pseudo-)circumcenters.
 \begin{corollary}
 	The following assertions hold.
 	  \begin{enumerate}
  	\item  Suppose that $\mathcal{H} =\mathbb{R}^{n}$, that $f$ is Legendre such that $\dom f^{*}$ is open, that $\overleftarrow{E}_{f}(\mathcal{S})  \neq \varnothing$,  that $f$	allows forward Bregman projections,   that $\aff (\mathcal{S}) \subseteq \inte \dom f$, and that $\nabla f (\aff (\mathcal{S})  ) $ is a closed  affine subspace.    Then $\left( \forall z \in \overleftarrow{E}_{f}(\mathcal{S}) \right)$ $ \overrightarrow{\Pro}^{f}_{\aff( \mathcal{S}  )} (z) \in \overleftarrow{\CCO{}}_{f}(\mathcal{S})$. $($See, \cref{theorem:formualCCS:Pleft}\cref{theorem:formualCCS:Pleft:P}.$)$
  	
  	\item  Suppose that $\overleftarrow{E}_{f}(\mathcal{S}) \neq \varnothing$ and that $\mathcal{S} \subseteq \inte \dom f $ and  $\aff(\nabla f( \mathcal{S} ) ) \subseteq \dom f$.  Then $\left( \forall z \in \overleftarrow{E}_{f}(\mathcal{S}) \right)$ $\overleftarrow{\CCO{}}^{ps}_{f}(\mathcal{S})=\Pro_{\aff(\nabla f( \mathcal{S} ) )} (z)  $.  $($See, \cref{theorem:formualCCS}\cref{theorem:formualCCS:EucP}.$)$
  	
  	\item Suppose that $\mathcal{S} \subseteq \inte \dom f $ and $\aff ( \nabla f (\mathcal{S}) ) \subseteq \dom f$, and that $\nabla f(q_{0}), \nabla f(q_{1}), \ldots, \nabla f(q_{m})$ are affinely independent.
  	Then $\overleftarrow{\CCO{}}^{ps}_{f}(\mathcal{S}) $ uniquely exists and has the explicit formula \cref{eq:thm:unique:LinIndpPformula}. $($See,  \cref{thm:unique:LinIndpPformula}.$)$
  	
  	\item Suppose that $f$ is Legendre,  that $ \aff (\mathcal{S}) \cap  \inte \dom f \neq \varnothing$, and that $\overrightarrow{E}_{f}(\mathcal{S} )   \neq \varnothing$.  Then
  	$\left( \forall z \in \overrightarrow{E}_{f}(\mathcal{S} )  \right)$
  	 $ \overleftarrow{\Pro}^{f}_{\aff(\mathcal{S})}(z) \in \overrightarrow{\CCO{}}_{f}(\mathcal{S})$.  $($See, \cref{theorem:forwardCCS}\cref{theorem:forwardCCS:P}.$)$
  	
  \item  Suppose that $f$ is Legendre, and that  $\aff( \mathcal{S}  ) \subseteq  \inte \dom f^{*}$ and $\overrightarrow{E}_{f}(\mathcal{S} )   \neq \varnothing$.  Then
  $\left( \forall z \in \overrightarrow{E}_{f}(\mathcal{S} )  \right)$  $\overrightarrow{\CCO{}}^{ps}_{f}(\mathcal{S})= \nabla f^{*} \left( \Pro_{\aff( \mathcal{S}  )} (\nabla f(z)) \right)$. $($See, 	\cref{theorem:psuCCS:forward}\cref{theorem:psuCCS:forward:P}.$)$

  \item
  Suppose that $f$ is Legendre, that $\aff ( \mathcal{S}) \subseteq \inte \dom f^{*}$, and  that $q_{0}, q_{1}, \ldots, q_{m}$ are affinely independent.
  Then $\overrightarrow{\CCO{}}^{ps}_{f}(\mathcal{S}) $ uniquely exists and has the explicit formula \cref{EQ:thm:LinIndpPformula:TpseudoCCS:formula}. $($See, \cref{thm:LinIndpPformula:TpseudoCCS}.$)$
  \end{enumerate}
 \end{corollary}

Note that backward (forward) Bregman pseudo-cucumcenters are nonempty whenever $\overleftarrow{E}_{f}(\mathcal{S})  \neq \varnothing$ (resp. $\overrightarrow{E}_{f}(\mathcal{S} )   \neq \varnothing$),
see \cref{cor:equi:Ef:CCfS} (resp. \cref{cor:equi:Ef:CCfS:right}).

Let 
 $\CCO{}$ be the classical circumcenter operator defined in \cite[Definition~3.4]{BOyW2018} under the Euclidean distance,
i.e., $f:= \frac{1}{2} \norm{\cdot}^{2}$. Then all backward and forward Bregman  (pseudo-)circumcenters
reduce to the classical circumcenter.

\begin{corollary} \label{cor:characterCCS:right}
	Suppose that $ f :=\frac{1}{2} \norm{\cdot}^{2}$.
	Then the following statements hold.
	\begin{enumerate}
		\item \label{cor:characterCCS:right:eq}$\CCO (\mathcal{S}) =\overleftarrow{\CCO{}}_{f}(\mathcal{S}) =\overleftarrow{\CCO{}}^{ps}_{f}(\mathcal{S}) =\overrightarrow{\CCO{}}_{f}(\mathcal{S}) =\overrightarrow{\CCO{}}^{ps}_{f}(\mathcal{S})$.
		\item \label{cor:characterCCS:right:norm} $\overleftarrow{E}_{f}(\mathcal{S})  = \overrightarrow{E}_{f}(\mathcal{S})  =\Big\{ y \in \mathcal{H} ~:~  (\forall i \in \I) ~ \norm{q_{i} -y} =\norm{q_{0} -y} \Big\}$.
		Consequently, $\CCO  (\mathcal{S})  \neq \varnothing$ if and only if there exists $x \in \mathcal{H}$ such that $	\norm{x -q_{0}} =\norm{x -q_{1}} =\cdots =\norm{x -q_{m}}$, that is, 	 $q_{0},q_{1} \ldots, q_{m}$ lie on a sphere with center $x \in \mathcal{H}$.
	\end{enumerate}
\end{corollary}


The following example illustrates the existence of the backward Bregman pseudo-circumcenter, while the classical circumcenter
does not exist.
\begin{example} \label{example:Existence}
	Suppose that $\mathcal{H} =\mathbb{R}^{3}$	and   $\mathcal{S}:= \{(1,2,1), (0.5,1.5,0.5), (1.5,2.5,1.5)\} $. Denote by  $(\forall x \in \mathbb{R}^{3} )$ $x= (x_{i})^{3}_{i=1} $.   Suppose that $(\forall x \in \left]0, +\infty\right[^{3} )$ $f (x)=   -\sum^{3}_{i=1}   \ln (x_{i})  $, with $\dom f =\left]0, +\infty\right[^{3} $. Then the following assertions hold.
	\begin{enumerate}
		\item  \label{example:Existence:no} The circumcenter under the Euclidean distance doesn't exist.
		\item  \label{example:Existence:yes} $\overleftarrow{\CCO{}}^{ps}_{f}(\mathcal{S}) \approx ( 0.7641,    0.8744,    0.7641)$.
	\end{enumerate}
\end{example}
\begin{proof}
	Denote by $x:=(1,2,1)$, $y:= (0.5,1.5,0.5)$, and $z:=(1.5,2.5,1.5)$.
	
	\cref{example:Existence:no}: Because $z-x = - (y-x)$, i.e., $x, y$ and $z$ are affinely dependent, due to \cite[Theorem~8.1]{BOyW2018}, \cref{example:Existence:no} is true.
	
	\cref{example:Existence:yes}: Because $(\forall u \in \left]0, +\infty\right[^{3} )$ $\nabla f(u) = \left(  -\frac{1}{u_{1}} ,   -\frac{1}{u_{2}},   -\frac{1}{u_{3}} \right)^{\intercal} $, we know that  $\nabla f(x) =-(1, \frac{1}{2}, 1)^{\intercal} $, $\nabla f(y) =-(2, \frac{2}{3}, 2)^{\intercal} $, and $\nabla f(z) =-(\frac{2}{3}, \frac{2}{5}, \frac{2}{3})^{\intercal} $. It is easy to see that $\nabla f(x) $,  $\nabla f(y) $, and  $\nabla f(z) $ are affinely independent. Hence, the desired result follows easily from \cref{prop:formualCCS:matrixEQ}.
\end{proof}
\cref{fig:BackwardBregmanpseudoCCS} below illustrates \cref{example:Existence}.
The red intersection point of the three green, blue and yellow Bregman balls, located on the affine subspace $\aff \nabla f (\mathcal{S})$ and labeled  as ps-CC(S),  is $\overleftarrow{\CCO{}}^{ps}_{f}(\mathcal{S})$.
However,  the points $x, y$ and $z$ in \cref{fig:BackwardBregmanpseudoCCS} are colinear, which implies that the classical circumcenter doesn't exist.
\begin{figure}[H]
	\centering
	\includegraphics[width=0.5\textwidth]{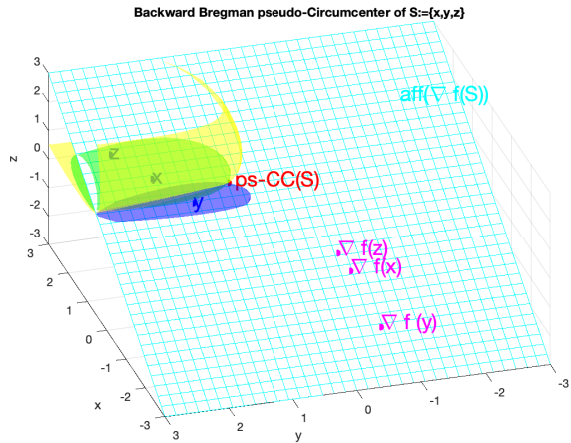}
	\caption{Backward Bregman pseudo-circumcenter exists but the classical circumcenter doesn't.}
	\label{fig:BackwardBregmanpseudoCCS}
\end{figure}

The following \cref{example:CCS}\cref{example:CCS:Consist} and \cref{example:CCS}\cref{example:CCS:Different} show, respectively, the difference of the backward and forward Bregman (pseudo-)circumcenters, and the classical circumcenter.
\begin{example} \label{example:CCS}
	Suppose that $\mathcal{H} =\mathbb{R}^{3}$	and   $\mathcal{S}:= \{(1,1,1), (1,2,1), (1,1,2)\} $. Denote by  $(\forall x \in \mathbb{R}^{3} )$ $x= (x_{i})^{3}_{i=1} $.   We consider backward and forward Bregman (pseudo-)circumcenters of $\mathcal{S}$ w.r.t. the energy and negative entropy functions.
	\begin{enumerate}
		\item \label{example:CCS:Consist} If $ f :=\frac{1}{2} \norm{\cdot}^{2}$, in view of \cref{cor:characterCCS:right}\cref{cor:characterCCS:right:eq}, we have $\overleftarrow{\CCO{}}_{f}(\mathcal{S}) =\overrightarrow{\CCO{}}_{f}(\mathcal{S}) =\overleftarrow{\CCO{}}^{ps}_{f}(\mathcal{S})  =\overrightarrow{\CCO{}}^{ps}_{f}(\mathcal{S})=\CCO{(\mathcal{S})} = (1, \frac{3}{2}, \frac{3}{2})^{\intercal}$.
		\item \label{example:CCS:Different} Suppose that $(\forall x \in \left[0, +\infty\right[^{3} )$ $f(x) =\sum^{3}_{i=1} x_{i} \ln( x_{i}) -x_{i}$. Then $f(1,1,1)=-3$, $f(1,2,1)=2\ln(2) -4$, $f(1,1,2)=2\ln(2) -4$. Moreover, $(\forall x \in \left]0, +\infty\right[^{3} )$ $\nabla f(x) = \left(  \ln(x_{1}),   \ln(x_{2}),  \ln(x_{3}) \right)^{\intercal} $, and, in view of \cite[Proposition~13.30]{BC2017} and \cite[Example~6.5]{BB1997Legendre}, $(\forall x \in \mathbb{R}^{3} )$ $f^{*}(x) =\sum^{3}_{i=1} \exp (x_{i} )$ and $\nabla f^{*}(x) =\left( \exp (x_{1} ), \exp (x_{2} ), \exp (x_{3} )  \right)^{\intercal}$.

		\begin{enumerate}
			\item \label{example:CCS:Different:a} \cref{exam:Back:Rn}\cref{exam:Back:Rn:negative:}  and  \cref{exam:forward:Rn}\cref{exam:forward:Rn:negative:} imply that
			\begin{align*}
			\overleftarrow{\CCO{}}_{f}(\mathcal{S})  = \left\{ \left(1, \frac{1}{\ln 2} , \frac{1}{\ln 2} \right) \right\} \quad  \text{and} \quad
			\overrightarrow{\CCO{}}_{f}(\mathcal{S}) = 	\left\{ \left(1, \frac{4}{\rm e} , \frac{4}{\rm e} \right) \right\}.
			\end{align*}
			\item
			\cref{thm:unique:LinIndpPformula} and  \cref{thm:LinIndpPformula:TpseudoCCS} imply that
			\begin{align*}
			\overleftarrow{\CCO{}}^{ps}_{f}(\mathcal{S})  = \left(0, \frac{1}{\ln2}, \frac{1}{\ln2} \right) \quad \text{and} \quad \overrightarrow{\CCO{}}^{ps}_{f}(\mathcal{S}) =\left( {\rm e}, \frac{4}{\rm e},   \frac{4}{\rm e} \right).
			\end{align*}
		\end{enumerate}
	\end{enumerate}
\end{example}
To visualize \cref{example:CCS}\cref{example:CCS:Different}, put $x:=(1,1,1), y:= (1,2,1)$, and $z:= (1,1,2)$ so that $\mathcal{S}=\{x,y,z\}$. \cref{fig:BFBregmanCircumcenters} below illustrates the difference between sets $\overleftarrow{E}_{f}(\mathcal{S})$ and $\overrightarrow{E}_{f}(\mathcal{S})$ because of the asymmetry of the general Bregman distance.
The red and magenta lines are $\overleftarrow{E}_{f}(\mathcal{S})$ and $\overrightarrow{E}_{f}(\mathcal{S})$, respectively.
The $\overleftarrow{\CCO{}}_{f}(\mathcal{S})$ (resp. $\overrightarrow{\CCO{}}_{f}(\mathcal{S})$) is the intersection  of the red (resp. magenta) line and the cyan  plane.
\begin{figure}[H]
	\centering
	\includegraphics[width=0.5\textwidth]{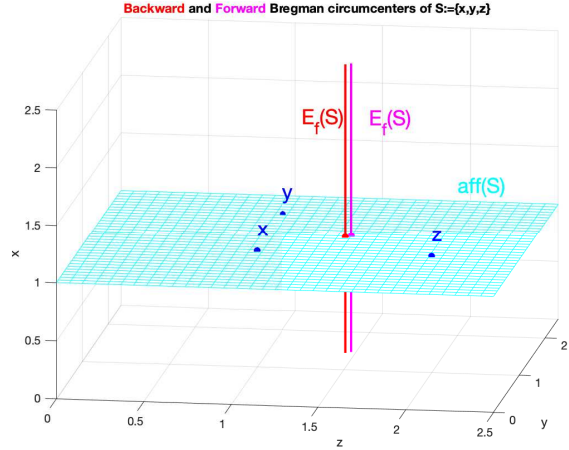}
	\caption{Difference of the backward and forward Bregman circumcenters}
	\label{fig:BFBregmanCircumcenters}
\end{figure}
In \cref{fig:BFBregmanBalls}(a)
(resp.\cref{fig:BFBregmanBalls}(b)), the green, blue and yellow sets are \enquote{Bregman balls}.
Because the intersections of the corresponding  three balls in \cref{fig:BFBregmanBalls} happen to be located  on the affine subspace $\aff \{x,y,z\}$, via \cref{defn:CCS:Bregman:left}\cref{defn:CCS:Bregman:left:}  and \cref{defn:CCS:Bregman:forward}\cref{defn:CCS:Bregman:forward:},
these red intersection points labelled as $CC(S)$ on the left and right pictures  below  are $\overleftarrow{\CCO{}}_{f}(\mathcal{S})$
and $\overrightarrow{\CCO{}}_{f}(\mathcal{S})$, respectively.

\begin{figure}[H]
	\centering
	\noindent
	\hspace*{-0.8cm}
	\begin{tabular}{ccc}
		\subfloat[]{\includegraphics[scale=0.45]{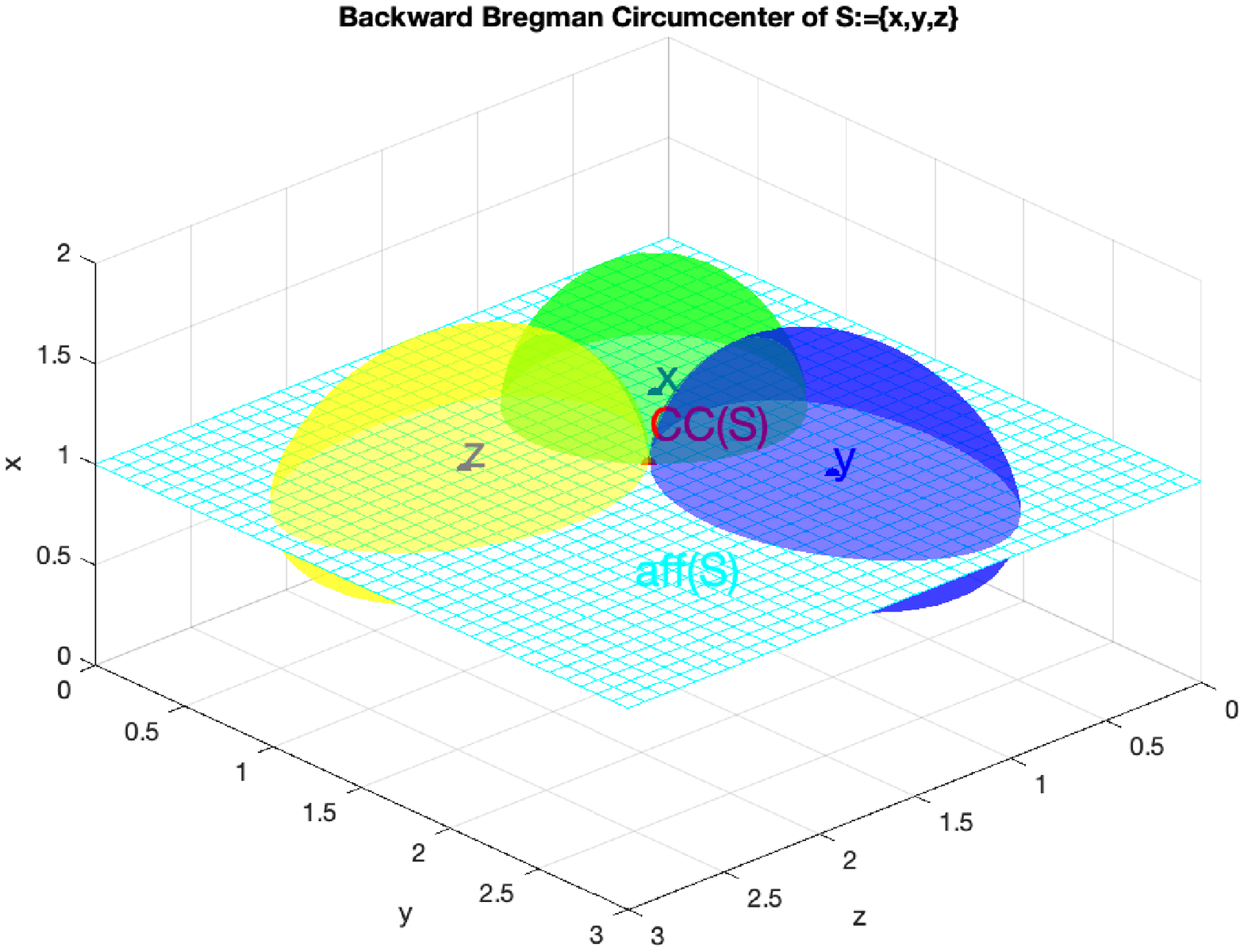}}
		\hspace*{0.2cm}
		\subfloat[]{\includegraphics[scale=0.45]{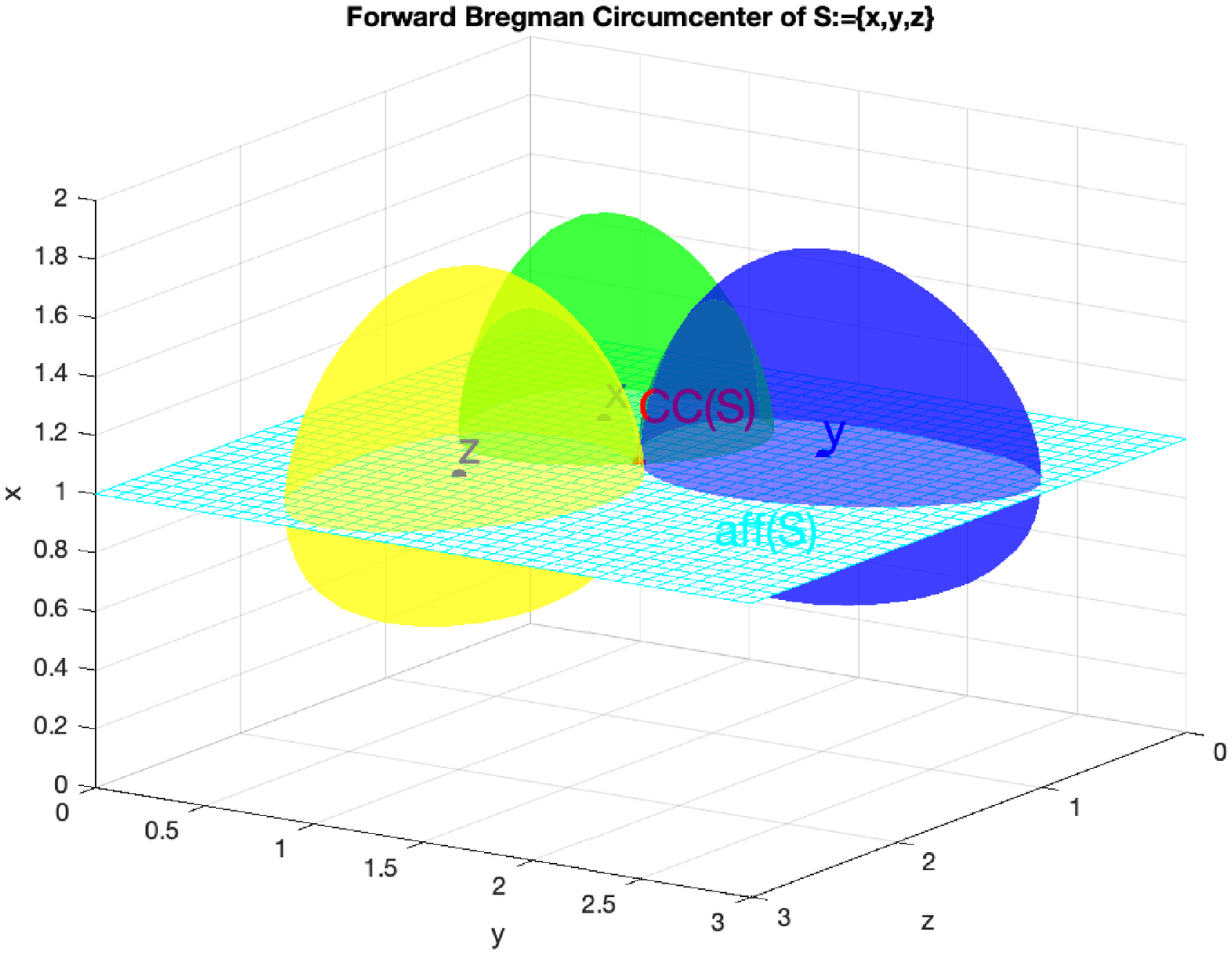}}
	\end{tabular}
	\caption{Backward and forward {Bregman} circumcenters} \label{fig:BFBregmanBalls}
	\hspace*{-0.5cm}
\end{figure}

\section{Conclusions}\label{s:conclude}
 In this work, we presented basic theory of Bregman circumcenters. We introduced  backward and forward Bregman
 (pseudo-)circumcenters. We also explored the existence and uniqueness of backward and forward Bregman (pseudo-)circumcenters, and presented explicit formulae of Bregman backward and forward pseudo-circumcenters. Various examples were given to illustrate these Bregman circumcenters.
 The connections between backward and forward Bregman (pseudo-)circumcenters were also established
 by duality.
We shall pursue applications of Bregman circumcenters in accelerating iterative methods in optimization in future.

	\section*{Acknowledgements}
XW was partially supported by the NSERC Discovery Grant.	
The authors would like to thank Dr. H.~H.~Bauschke for
his useful discussions and suggestions.

\addcontentsline{toc}{section}{References}

\bibliographystyle{abbrv}

\end{document}